\documentclass[twoside,11pt]{article} 
\usepackage{jmlr2e_cmTECHREP}

\usepackage{bm}
\usepackage{amsmath}

\usepackage[backend=biber,style=numeric, firstinits=true, maxalphanames=4, maxnames=99,minnames=99]{biblatex}
\renewbibmacro{in:}{}
\DeclareSourcemap{
  \maps[datatype=bibtex]{
    \map{
      \step[fieldsource=journal, match={Numerische Mathematik}, replace={Numer. Math.} ]
    }
     \map{
      \step[fieldsource=journal, 
     match={SIAM Journal on Numerical Analysis}, replace={SIAM J. Numer. Anal.} ]
    }
    \map{
      \step[fieldsource=journal, 
     match={Journal of Computational Physics}, replace={J. Comput. Phys.} ]
    } 
  }
}

\usepackage{breakurl}
\usepackage{mathrsfs }

\usepackage{amsmath,amssymb,amsxtra,comment,graphicx,psfrag}
\usepackage{bm,mathrsfs}
\usepackage{mathtools}
\usepackage{xspace}
\usepackage{stmaryrd}

\usepackage{enumerate,comment,graphicx,psfrag}

\usepackage{hhline}

\usepackage{bookmark}
\usepackage{nameref}

\usepackage{verbatim}
\usepackage{tabls}

\usepackage{fancyhdr}
\usepackage{epsfig}
\usepackage{epsfig,epstopdf}

\usepackage{pstricks,pst-plot,pst-func}
\usepackage{pspicture}
\usepackage{curves}

\usepackage{subcaption}
\usepackage{bbm}

\include{rgb}
\def\dimx{\nu }

\def\div{\text{\rm div}}
 
\def\cal \mathcal{ }

\def\nst2{\| _*} 
 
\def\a12{A_h ^{1/2} } 
\def\d{{\mathrm d}}
\def\tr|{|\! |\! |}

\def\C{{\mathbb C}}
\def\R {{\mathbb R}}

\def\L{{\mathscr{L}}}

\def \a{\alpha }

\def\T_h{{{\mathcal T}_h}}

\def\<{{\langle }}
\def\>{{\rangle }}

\newcommand{\qp}[1]{\ensuremath{\!\left({#1}\right)}}

 \newcommand{\norm}[1]{\ensuremath{\left|#1\right|}}
  \newcommand{\Norm}[1]{\ensuremath{\left\|#1\right\|}}
  
  \renewcommand{\vec}[1]{\boldsymbol{#1}}

\def\normVB#1{\|#1\|_{\mathscr {V}_{BC}}}

\def\Hnod1{\stackrel{\!\scriptscriptstyle 0}{H^1}\!\!}

\def\ubar{\overline{u}}
\def\vbar{\overline{v}}

\def\Bc{{\mathcal B}}

\def\Hc{{\mathscr H}}

\def\Vc{{\mathscr V}}
\def\Vcbc{{\mathscr V}_{BC}}

\DeclareSymbolFont{matha}{OML}{txmi}{m}{it}
\DeclareMathSymbol{\varv}{\mathbf}{matha}{118}

\def\L{{\mathscr{L}}}

\newtheorem{assumption}{Assumption}[section]

\newcommand\restr[2]{{% we make the whole thing an ordinary symbol
  \left.\kern-\nulldelimiterspace % automatically resize the bar with \right
  #1 % the function
  \vphantom{\big|} % pretend it's a little taller at normal size
  \right|_{#2} % this is the delimiter
  }}

\NewDocumentCommand{\dgal}{sO{}m}{%
  \IfBooleanTF{#1}
    {\dgalext{#3}}
    {\dgalx[#2]{#3}}%
}

\NewDocumentCommand{\dgalext}{m}{%
  \sbox0{%
    \mathsurround=0pt % just for safety
    $\left\{\vphantom{#1}\right.\kern-\nulldelimiterspace$%
  }%
  \sbox2{\{}%
  \ifdim\ht0=\ht2
    \{\kern-.625\wd2 \{#1\}\kern-.625\wd2 \}%
  \else
    \left\{\kern-.7\wd0\left\{#1\right\}\kern-.7\wd0\right\}%
  \fi
}

\NewDocumentCommand{\dgalx}{om}{%
  \sbox0{\mathsurround=0pt$#1\{$}%
  \sbox2{\{}%
  \ifdim\ht0=\ht2
    \{\kern-.625\wd2 \{#2\}\kern-.625\wd2 \}%
  \else
    \mathopen{#1\{\kern-.7\wd0 #1\{}
    #2
    \mathclose{#1\}\kern-.7\wd0 #1\}}
  \fi
}

\renewcommand{\R}{{\mathbb{R}}}

\renewcommand{\T}{\mathbb{T}^d}

\makeatletter
\renewcommand{\@makefntext}[1]{\noindent#1} % Remove indentation and dot
\makeatother

\begin{document}

\title
{ Variational Principles for the Helmholtz equation: application to  Finite Element and Neural Network approximations} 

\author{ \name{\begin{center}{\sc  G. Makrakis $^{1,2}$, C.\  Makridakis $^{2, 3},$ D.\ Mitsoudis$^{ 2, 4},$ \\ M.\ Plexousakis $^{1,2}$   \and T.\ Pryer  $^{5 }$    }\end{center}} }
\footnote  {$^{1}$ DMAM, University of Crete, Greece \\
 $^{2 }$ IACM-FORTH, Greece\\
         $^{3 }$ MPS, University of Sussex, United Kingdom\\
         $^{4 }$ DNA, University of West Attica, Greece\\
         $^{5 }$ DMS and Institute of Mathematical Innovation, University of Bath,  United Kingdom}

\editor{}
\maketitle

\abstract{
    In this paper, we investigate whether variational principles can be
  associated with the Helmholtz equation subject to impedance
  (absorbing) boundary conditions. This model has been extensively
  studied in the literature from both mathematical and computational
  perspectives. It is classical with wide applications, yet accurate
  approximation at high wavenumbers remains challenging.

  We address the question of whether there exist energy functionals
  with a clear physical interpretation whose stationary points,
  the zeros of their first variation, correspond to solutions of the
  Helmholtz problem. Starting from Hamilton's principle for the wave
  equation, we derive time-harmonic energies. The resulting
  functionals are generally indefinite. As a next step, we
  construct strongly coercive augmentations of these indefinite
  functionals that preserve their physical interpretation. Finally, we
  show how these variational principles lead to practical numerical
  methods based on finite element spaces and neural network
  architectures.
}

\section{Introduction}\label{Se:1}

\subsubsection*{Helmholtz equation and Variational Principles}\label{Se:1NN}
 {The Helmholtz equation with an impedance %(absorbing) 
boundary condition }
\begin{equation}\label{Helm_PDE}
  -\Delta u - k^2 u = f\quad \text{ in }  \varOmega 
  \qquad \partial_{\mathbf{n}} u = i k  u \quad \text{ on }  \partial\varOmega
\end{equation}
 {is the time-harmonic analogue of the wave equation with absorbing boundary conditions of first order. Here, the function $u$ is  complex-valued, 
 $u:\varOmega\to\mathbb{C}$, $ \varOmega\subset\mathbb{R}^{\dimx}$ is a bounded Lipschitz domain, $f\in L^2(\varOmega)$ is a complex-valued function and $k>0$ is the wavenumber.}
{We use the compact notation $\L u:=-\Delta u - k^2 u$.}
 {The use of various absorbing boundary conditions is motivated by the need to reduce wave propagation problems posed on large or unbounded domains to problems defined on computationally tractable bounded domains.} 
{While this model is classical and broadly applicable, the design of numerical algorithms capable of accurately approximating it at high wavenumbers remains a significant challenge.}

In this work, we investigate whether \emph{variational principles} can be associated with the Helmholtz equation and in particular with \eqref{Helm_PDE}.   {Specifically, we ask whether there exist energies motivated from first principles  whose \emph{stationary points}, the zeros of their first variation, coincide with solutions of \eqref{Helm_PDE}. }
To the best of our knowledge, such principles have not been systematically  {developed for impedance Helmholtz problems such as \eqref{Helm_PDE}.}

 {Starting from Hamilton's principle for the wave equation, we derive time-harmonic energies that are, in general, \emph{indefinite}. These energies apply naturally to the Helmholtz equation posed on unrestricted domains. When the domain is truncated and impedance boundary conditions are imposed, suitable modifications  become necessary. 
In unrestricted domains, the corresponding Lagrangian functionals may be chosen to be real-valued. However, to account for the influence of the artificial boundary introduced by the absorbing condition, we employ complex-valued Lagrangians. Nevertheless, the variational structure is preserved, since the stationarity condition has a unique critical point corresponding to the solution of \eqref{Helm_PDE}.}

 {We then construct strongly coercive augmentations of these energies. These regularised Lagrangians are carefully designed modifications of the original indefinite functionals and provide stability bounds in the case of star-shaped domains.}

 {For all formulations considered, we address existence and uniqueness questions for the corresponding variational problems by employing the complex version of the Lax-Milgram theorem. Throughout, our derivations and proposed formulations are motivated by the goal of remaining as close as possible to the underlying physical properties of the problem.}

Apart from its theoretical appeal, the derivation of new variational principles has significant practical implications. Modern methods that utilise neural-network-based discrete spaces typically rely on discrete minimisation problems, where the loss functional is constructed either by approximating a continuous energy associated with the problem or by using the $L^2$ residual of the differential operator \cite{Karniadakis:pinn:orig:2019}. While the residual-based approach is straightforward, it  {is desirable
to employ alternative methods, especially}
for PDEs with {highly oscillatory} behaviour {such as \eqref{Helm_PDE}}.  
  { The hope is that, by employing a variational principle with strong coercivity, one obtains a more robust foundation for efficient and reliable methods. In fact, }{we show that the proposed variational principle  {may lead to} an effective neural-network formulation for \eqref{Helm_PDE}. We also present prototype numerical evidence that the same principle leads to efficient finite element formulations and outline how its properties motivate discretisations with improved characteristics.} Strongly coercive bilinear forms for the Helmholtz equation were previously developed in \cite{Moiola_Spence2014}, leading to finite element methods that, in general, produce non-symmetric matrices. While we employ similar analytical tools, namely the Rellich and Morawetz identities, which are effective in the analysis of \eqref{Helm_PDE}, our motivation and design strategy differ substantially, see Remark \ref{remark_MS} and Section \ref{sec:design}.

\subsubsection*{Bibliography and Main Contributions}
The Helmholtz equation has been extensively investigated from multiple perspectives, including the treatment of various boundary conditions and alternative formulations; see, for example
\cite{MorawetzLudwig1968, 
FixMarin1978,
Goldstein1982,
BGT1985,
Aziz_et_al_1988,
KellerGivoli1989,
Givoli1992,
Douglas_et_al_1993,
perthame1999morrey,
Yafaev2000,
SingerTurkel2004,
CummingsFeng2006,
Mitsoudis2007,
Athanassoulis_et_al_2008,
chandler2008wave,
MitsoudisPlexousakis2009,
MMP2012,
SauterTorres2018,
GROTE_control_H_2019,
GrahamSauter2020}
{and references therein}.
Frequency-explicit stability estimates are important because, among other reasons, when combined with finite element error bounds, they highlight the constraints imposed by moderate to high {wavenumbers} and {mesh resolution}. These analyses reveal subtle correlations between the wavenumber and the mesh size, an aspect first systematically investigated within the research {programme} initiated by I. Babu\v ska and his collaborators (see, e.g., \cite{IhlenburgBabuska1995, Ihlenburg1998, BabuskaSauter2000}, and the references therein).
Stability analyses based on energy methods employing nonstandard test functions were developed in {\cite{Makridakis_et_al_1996} and \cite{Melenk1995}}. 
These test functions belong to a broader class of so-called Morawetz multipliers, introduced by Morawetz and Ludwig \cite{MorawetzLudwig1968} for the justification of geometric optics for the reduced wave equations. Such multipliers are instrumental in deriving Rellich-type identities, as used, for instance, in \cite{CummingsFeng2006, Spence_Chandler-Wilde_integral_2011, Moiola_Spence2014} and related works.
{The multipliers $\mathcal{M}$ typically involve $x\cdot\nabla u$ together with possible lower-order terms; the identities obtained by testing the PDE with $\mathcal{M}u$ are known as Morawetz identities.}
These identities play a crucial role in establishing stability estimates for a wide range of problems in both linear and nonlinear wave propagation, \cite{Morawetz_PIrish_1972,
Lesky2003,
MetcalfeEtal2005,
Tao_2006,
DAncona_2012,
Hintz_Zworski_2017,
Tao_MI_notices,
cossetti2024virial,
Bigniardi_Moiola_wave_2025}.
Such identities will also be instrumental in the design and analysis of the variational principles introduced in this paper. 

 {The main contributions of this paper are as follows. We derive variational principles for the Helmholtz equation \eqref{Helm_PDE} from first principles, starting from Hamilton's principle (least action) and a Lagrangian formulation of the wave equation to motivate the associated energy functionals. As expected, these functionals are generally indefinite. Incorporating impedance boundary conditions presents an additional challenge, which we address by employing complex-valued Lagrangians. We then construct regularised versions of these functionals that are strongly coercive.} % while preserving their physical relevance.
Our main idea is to perturb the physical Lagrangian by adding a term of the form 
$\gamma \| \L u-f \|^2$ and a corresponding boundary term. This regularisation enhances the stability of the variational principle, though in a nontrivial way. It turns out that, by {choosing $\gamma$ suitably}, one can obtain a coercive functional with respect to an appropriately strengthened norm. The resulting coercivity arises from a combination of the original physical Lagrangian and suitable Morawetz identities, which are employed to derive lower bounds first for the bulk energy terms and subsequently for the boundary contributions.
 {The resulting energies form the basis for well-posed variational principles and minimisation problems whose unique minimisers approximate the solutions of the Helmholtz equation.} 
  
 {Finally, we show how these variational principles yield discrete approximation schemes. We present %a unified framework that encompasses 
  computational methods based on classical conforming finite elements and neural-network discretisations. The primary focus is the mathematical analysis underpinning the design of the new variational principles, but we also include prototype formulations for both finite elements and neural networks to illustrate their potential for computing approximations to \eqref{Helm_PDE}.    It is worth emphasising that neural network approximations depend on appropriate minimisation principles, and, to our knowledge, this is a new coercive augmented variational formulation for impedance Helmholtz problems that can be used with neural trial spaces. Preliminary results suggest several desirable properties, a detailed numerical investigation is deferred to future work. }
Regarding the extensive literature on finite element schemes for approximating the Helmholtz equation, as mentioned above, analyses combining frequency-explicit stability bounds for weak solutions provide both theoretical and computational insights into the behaviour of such methods and the pollution effects that become particularly significant at high {wavenumbers} \cite{BabuskaSauter2000}. Classical error estimates and quasi-optimality bounds based on inf-sup conditions trace back to ideas introduced by Aziz et al. \cite{Aziz_et_al_1988}, Makridakis et al. \cite{Makridakis_et_al_1996} and Melenk \cite{Melenk1995}, which build upon the approach of Schatz \cite{Schatz1974} for indefinite problems.
Further detailed studies of the pollution effect and the improved performance of higher-order methods can be found in \cite{Sauter_refined, MelenkSauter2010}.  {Least squares methods were considered in \cite{First-orderLS_Helm_Lee_et_al_2000}, methods employing basis functions that solve local problems and thus achieve enhanced performance are discussed in \cite{HipMP_serveyT_2016}.} Discontinuous Galerkin methods were investigated in \cite{Feng_dg_2011}, while specially designed algorithms that avoid the pollution effect were proposed in \cite{Peterseim_polloutionHelm_2017}. Extensions to problems with variable coefficients have also been considered, for example, in \cite{GrahamSauter2020}.

The paper is organised as follows. {Section \ref{sec:energy} recalls the physical derivation of the wave equation via Hamilton's principle and motivates   natural Lagrangian functionals for the Helmholtz equation. Section \ref{sec:variational} introduces notation and functional-analytic preliminaries and defines the variational formulations and the regularised Lagrangians, and Section \ref{sec:proofs} proves the main results. Section \ref{sec:numerics} discusses discrete approximations based on neural networks and conforming finite elements and illustrates their formulation within the variational framework.

%%%%%%%%%%%%%%%%%%%%%SECTION 2%%%%%%%%%%%%%%%%
%%%%%%%%%%%%%%%%%%%%%%%%%%%%%%%%%%%%%%%%%%%
\section{Energy principles}\label{sec:energy}
\subsubsection*{The basic model: Transverse vibrations of a string}

%%%%%%%%%%%%%%%%WAVE EQUATION%%%%%%%%%%%%%%%%%%%
In this section, we explain how the wave and the Helmholtz equations can be derived from Hamilton's principle. Such a derivation provides us with a natural Lagrangian, which in turn arises in the standard variational reformulation of certain boundary value problems for the Helmholtz equation. For simplicity, we restrict the discussion to the case $\dimx = 1$, and we consider the transverse motion of a string occupying the interval {$\Omega=(\alpha,\beta)$} along the $x$-axis. Analogous results can be derived for $\dimx = 2$ or $\dimx = 3$, corresponding to the transverse motion of a membrane or to sound waves in an ideal compressible fluid, respectively.
We assume that the mass density (per unit length) of the string is $\rho(x)$ and that the string is subjected to tension by opposite forces of strength $S$ at its ends. Then, the transverse deflection $U(x,t)$ satisfies the wave equation
\begin{equation} \label{eq:wave-nh}
\frac{1}{c^2(x)}\frac{\partial^2 U}{\partial t^2}=\frac{\partial^2 U}{\partial x^2}  {- \frac{F(x,t)}{S}}, \ \  t>0\ , \ \  \alpha < x< \beta,
\end{equation}
where $c(x)= \left(S/\rho(x)\right)^{1/2}$ is the phase velocity of the waves propagating along the string, and  {$F(x,t)$ denotes an external transverse load per unit length (the sign convention in \eqref{eq:wave-nh} is immaterial for what follows)}.  {We assume that $c(x)$ and $F(x,t)$ are sufficiently smooth functions.}

At the end points $x=\alpha,\beta$,  it is necessary to impose boundary conditions, usually in the form of linear combination of $U(x,t)$ and $\partial U/\partial x (x,t)$,  for all $t>0$, and also we must prescribe initial data for $U(x,0)$ and $\partial U/\partial t(x,0)$, for all $\alpha < x< \beta$, in order   {to be able to construct} a unique solution of the problem.

In many applications, as well as for fundamental theoretical reasons in scattering theory, it is necessary to study the corresponding time-harmonic problem, and, in particular, the associated Green's function. To this end, we consider that the motion of the string starts from rest at time $t = 0$,  and it subsequently evolves  {under the action} of a transverse force distribution which varies periodically in time with frequency $\omega$. Such a force is represented by the distribution $F(x,t)=\Re\bigl(e^{-i\omega t} f(x)\bigr) H(t)$, where, typically, $f(x)$ is a localised function and $H(t)$ is the  {Heaviside function}.  {If the Limiting Amplitude Principle applies, then as $t \to \infty$ the motion becomes time-harmonic of the form} $U(x,t) = \Re\bigl(u(x) e^{-i\omega t}\bigr)$, where the complex amplitude $u(x)$ satisfies the Helmholtz equation
\begin{equation}\label{eq:helm-nh}
u''(x) + k^2(x) u(x) = \frac{f(x)}{S}, \qquad  k(x) =\frac{\omega}{c(x)},   \quad \alpha < x< \beta \ ,
\end{equation}
and certain boundary conditions dictated by $U(x,t)$ and $\partial U/\partial x (x,t)$ at $x=\alpha,\beta$. Whenever $\alpha=-\infty$ and/or $\beta=+\infty$, the boundary conditions must be replaced by  {outgoing (Sommerfeld-type) radiation conditions} to ensure uniqueness of the solution.

\subsubsection*{Variational derivation of the wave equation}

Let
\begin{equation}
\mathcal{T}(x,t)= \frac12  \rho(x) \Bigl(\frac{\partial U}{\partial t}\Bigr)^2  
\end{equation}
be the \emph{kinetic energy density} (that is, energy per unit length), and
\begin{equation}
\mathcal{V}(x,t)=  {\frac12   S  \Bigl(\frac{\partial U}{\partial x}\Bigr)^2}
\end{equation}
be the \emph{potential energy density} of the string {; the external loading enters through the work term in the Lagrangian rather than as stored potential energy}.
For any time interval $0<t<T$, $T<\infty$, the \emph{action} of the system is given by the space-time integral
\begin{equation}\label{eq:dvar}
\mathcal{I}(U)= \frac{1}{T}\int_{0}^{T}\int_{\alpha}^{\beta}  \mathcal{L}\Bigl(x, {U},\frac{\partial U}{\partial t},\frac{\partial U}{\partial x}\Bigr)  \d x   \d t \ ,
\end{equation}
of the \emph{Lagrangian density}
\begin{equation}\label{eq:lagrdens}
 \begin{split}
   \mathcal{L}\Bigl(x, {U},\frac{\partial U}{\partial t},\frac{\partial U}{\partial x}\Bigr)
   &= \mathcal{T}(x,t)-\mathcal{V}(x,t) { - F(x,t) U(x,t)}
   \\
   &= \frac12  \rho(x) \Bigl(\frac{\partial U}{\partial t}\Bigr)^2
   - \frac12   S  \Bigl(\frac{\partial U}{\partial x}\Bigr)^2  { - F(x,t) U(x,t)} \ .
 \end{split}
\end{equation}
According to \emph{Hamilton's variational principle} \cite{Lanczos_VP_1949,Gelfand_Fomin_CV}, the actual motion of the string follows from the stationarity of the action, i.e., the derivative of the action with respect to the field $U$ in the direction of any appropriate admissible function $\eta=\eta(x,t)$ \cite[Section 8.1]{Evans},
\begin{equation}\label{eq:deul1}
 {D_{U}\, \mathcal{I}{(U)}[\eta]=\frac{d}{d\epsilon}I(U+\epsilon \eta)|_{\epsilon=0}=0} \
\end{equation}

Equation \eqref{eq:deul1} implies
\begin{equation}\label{eq:lagrange}
 {D_{U}\, \mathcal{I}(U)}[\eta]=\frac{1}{T} \,  \int_{0}^{T}\int_{\alpha}^{\beta} \frac{\delta {\mathcal L}}{\delta U}\, \eta \, dx \, dt
+ \frac{1}{T}\,\int_{0}^{T}\int_{\alpha}^{\beta} \, \frac{\partial}{\partial t} \Biggl(\frac{\partial {\mathcal L}}{\partial \Bigl(\frac{\partial U}{\partial t}\Bigr)} \, \eta \Biggr)+
\frac{\partial}{\partial x} \Biggl(\frac{\partial {\mathcal L}}{\partial \Bigl(\frac{\partial U}{\partial x}\Bigr)} \, \eta \Biggr) \, dx \, dt =0\ ,
\end{equation}
where$\frac{\delta {\mathcal L}}{\delta U}$ is the variational derivative
\begin{equation}\label{eq:varder}
\frac{\delta {\mathcal L}}{\delta U}= \, \frac{\partial {\mathcal L}}{\partial U}
-\frac{\partial }{\partial t}\Biggl(  \frac{\partial {\mathcal L}}{\partial \Bigl( \frac{\partial U}{\partial t}\Bigr)}  \Biggr) - \frac{\partial }{\partial x}\Biggl( \frac{\partial {\mathcal L}}{\partial
\Bigl(\frac{\partial U}{\partial x} \Bigr)} \Biggr) \ ,
\end{equation}
For the Lagrangian density \eqref{eq:lagrdens}, equation \eqref{eq:lagrange} gives the wave equation
\begin{equation}\label{eq:wave}
\frac{1}{c^2(x)} \frac{\partial^2 U}{\partial t^2}-\frac{\partial^2 U}{\partial x^2}= {- \frac{F}{S}} \ ,
\end{equation}
assuming that the boundary term
\begin{equation}\label{eq:deul2}
{\mathcal{R}}=\frac{1}{T}\,\int_{0}^{T}\int_{\alpha}^{\beta} \, \Biggl\{ \frac{\partial}{\partial t} \Biggl(\frac{\partial {\mathcal L}}{\partial \Bigl(\frac{\partial U}{\partial t}\Bigr)} \, \eta\Biggr)+
\frac{\partial}{\partial x} \Biggl(\frac{\partial {\mathcal L}}{\partial \Bigl(\frac{\partial U}{\partial x}\Bigr)} \, \eta \Biggr)\Biggr\}\, dx\, dt
\end{equation}
vanishes.

By integrating by parts, the boundary term \eqref{eq:deul2} is written as follows
\begin{eqnarray}\label{eq:remainder1}
{ \mathcal{R}}=\frac{1}{T}\, \int_{\alpha}^{\beta}\rho \Biggl( U_t(x,T)\eta(x,T)-U_t(x,0) \eta(x,0)\Biggr) \, dx \nonumber \\
-S\, \frac{1}{T}\,\int_{0}^{T}\Biggl(U_x(\beta, \tau)\eta(\beta, \tau)-U_x(\alpha, \tau)\eta(\alpha, \tau)\Biggr)\, d\tau  \ .
\end{eqnarray}

If we choose  $\eta$ to have compact support in $G=(\alpha, \beta)\times(0,T)$, then $\mathcal{R}\equiv 0$, and \eqref{eq:wave} holds in the sense of distributions on $G$.
 {Formally, since in this case we do not vary  $U$ on $\partial\, G$, we can handle inhomogeneous initial data, as well as impedance boundary conditions of the form $U_{t}=\pm \mu \, U_{x}$ for some $\mu >0$.
Moreover, if we assume that the admissible function $\eta$ satisfies $\eta(x,0)=\eta(x,T)=0$ for all $\alpha \le x \le \beta$, that is we do not vary $U(x,t)$ at the initial and final time, then  we must impose at $x=\alpha \ , \beta$, the natural Neumann (zero slope) boundary condition $\frac{\partial U}{\partial x}=0$,   to ensure that ${\mathcal{R}}=0$. }

\subsubsection*{Variational derivation of the Helmholtz equation}

The Helmholtz equation can be derived from the time-harmonic version of Hamilton's principle. We choose $T=\frac{2\pi}{\omega}$, the period of the time-harmonic wave, and define the stationary density $\ell$ as
the time average of the time-dependent Lagrangian density $ \mathcal{L}$,
\begin{equation}\label{eq:meanlang}
\ell\Bigl(x,u,\tfrac{\partial u}{\partial x}\Bigr)
:= \frac{2}{T}\int_{0}^{T} \mathcal{L}\Bigl(x, {U},\frac{\partial U}{\partial t},\frac{\partial U}{\partial x}\Bigr)  \d t \ ,
\end{equation}
where we substitute the time-harmonic ansatz $U(x,t)=\Re\bigl(u(x) e^{-i\omega t}\bigr)$, and $F(x,t)=\Re\bigl(f(x) e^{-i\omega t}\bigr)$.

After time integration in \eqref{eq:meanlang}, we derive
\begin{equation}\label{eq:hlagrdens}
\ell\Bigl(x,u,\tfrac{\partial u}{\partial x}\Bigr)
=  {\frac 12 \, S\, \Bigl(k(x)^2 |u|^2- \bigl|\tfrac{\partial u}{\partial x}\bigr|^2\Bigr) -  \Re\bigl(f(x) \overline{u(x)}\bigr),}
\qquad k(x)=\frac{\omega}{c(x)}.
\end{equation}
 {(Any positive constant multiple of $\ell$ produces the same stationary points, but the coefficients in \eqref{eq:hlagrdens} match the time average of \eqref{eq:lagrdens}.)}

The Helmholtz equation is derived from the stationarity of the action
\begin{equation}
\mathcal{J} (u) := \int_{\alpha}^{\beta} \ell\Bigl(x,u,\tfrac{\partial u}{\partial x}\Bigr)  \d x .
\end{equation}
This action arises naturally by rewriting the action $\mathcal{I}$ in \eqref{eq:dvar} in terms of the time average \eqref{eq:meanlang},
\begin{align}\label{eq:dvarh}
\mathcal{I} =  \frac{1}{T}\int_{0}^{T}\int_{\alpha}^{\beta} \mathcal{L}\Bigl(x, {U},\frac{\partial U}{\partial t},\frac{\partial U}{\partial x}\Bigr)  \d x  \d t
=  \frac12 \,\int_{\alpha}^{\beta} \ell\Bigl(x,u,\tfrac{\partial u}{\partial x}\Bigr)  \d x .
\end{align}

The use of the action $\mathcal{J}$ for the variational derivation of the Helmholtz equation is not standard, and it has been inspired by the variational derivation of the Schr\"odinger equation from a real Lagrangian \cite{Mo_book}, since both the  Schr\"odinger wave function and the wave function $u$ in the Helmholtz equation are complex.

The time-harmonic version of Hamilton's principle dictates that
\begin{equation}\label{eq:svar_1}
 {D_{u}\, \mathcal{J}{(u)}[\theta]=\frac{d}{d\epsilon}\mathcal{J}{(u)}(u+\epsilon \theta)|_{\epsilon=0}=0} \ .
\end{equation}
For an admissible function $\theta=\theta(x)$ with compact support in $(\alpha,\beta)$,  this condition implies that the Helmholtz equation \eqref{eq:helm-nh} holds in the distributional sense.
Similarly to the variational derivation of the wave equation, for other admissible functions $\theta(x)$ we can    accommodate various boundary conditions.

  {With variations taken with respect to the real inner product on $H^1((\alpha,\beta);\mathbb{C})$, the Euler-Lagrange equation for \eqref{eq:svar_1} is precisely the Helmholtz equation in the sense of distributions. See Section~\ref{sec:variational} for the detailed computation }

\subsubsection*{Impedance boundary conditions}

The \emph{total energy density} (per unit length) of the string at time $t$ is
\begin{equation}\label{eq:endens}
\mathcal{E}(x,t)= \mathcal{T}(x,t)+\mathcal{V}(x,t)   .
\end{equation}
For $F\equiv 0$, the total energy $E(t)=\int_{\alpha}^{\beta} \mathcal{E}(x,t)  \d x$ is constant for suitable choices of initial and boundary conditions. In fact, by direct calculation and using \eqref{eq:wave-nh},
\begin{equation}\label{eq:dere}
\frac{\d E(t)}{\d t}
= \int_{\alpha}^{\beta}\Biggl( \frac12  \rho(x) \frac{\partial}{\partial t}\Bigl(\frac{\partial U}{\partial t}\Bigr)^2
+ \frac12   S  \frac{\partial}{\partial t}\Bigl(\frac{\partial U}{\partial x}\Bigr)^2 \Biggr)  \d x
= S \Biggl[\frac{\partial U}{\partial x} \frac{\partial U}{\partial t}\Biggr]_{x=\alpha}^{x=\beta}   .
\end{equation}
If, for example, either the deflection $U$ or the slope $\partial U/\partial x$ at the endpoints is zero for all $t>0$, then the last term in \eqref{eq:dere} vanishes.
For an infinite string, where $\alpha=-\infty$ and $\beta=+\infty$, the derivative $\partial U/\partial t$ vanishes at both ends due to the finite speed of propagation.

However, the boundary conditions
\begin{equation}\label{eq:time-imp}
\frac{\partial U}{\partial x}=-\mu \frac{\partial U}{\partial t}\quad \text{at } x=\beta,
\qquad
\frac{\partial U}{\partial x}=+\mu \frac{\partial U}{\partial t}\quad \text{at } x=\alpha,
\end{equation}
where $\mu>0$,
are dissipative, since   by substituting \eqref{eq:time-imp} into \eqref{eq:dere} we get
\[
 {\frac{\d E(t)}{\d t} = -\mu\, S\bigl((\partial_t U)^2(\beta,t)+(\partial_t U)^2(\alpha,t)\bigr)< 0,}
\]
 {unless $\partial_t U(\alpha,t)=\partial_t U(\beta,t)=0$.}
In the sequel we will  take the string tension force $S=1$.

For the time-harmonic problem, the boundary conditions \eqref{eq:time-imp} imply the following dissipative (impedance)  boundary conditions for the Helmholtz equation:
\begin{equation}\label{eq:time-harm-imp}
\frac{\d u}{\d x}= i \sigma_{\beta} k_{\beta} u \quad \text{at } x=\beta,
\qquad
\frac{\d u}{\d x}= - i \sigma_{\alpha} k_{\alpha} u \quad \text{at } x=\alpha,
\end{equation}
where $k_{\alpha}=\omega/c(\alpha)$, $k_{\beta}=\omega/c(\beta)$, and $\sigma_{\alpha}=\mu c(\alpha)$, $\sigma_{\beta}=\mu c(\beta)$.

  {In higher dimensions the Helmholtz equation is stated in a bounded domain  
$\Omega\subset\mathbb{R}^{\dimx}$, where $u$ represents the deflection of a membrane for ${\dimx}=2$, and, e.g., the pressure of a barotropic fluid for ${\dimx}=3$.   In these cases the impedance condition on the boundary has the form
\[
 {\partial_{\mathbf{n}} u} = i k \sigma u \quad \text{on } \partial\Omega,
\]
 with $\mathbf{n}$ the outward unit normal and $\sigma>0$. The canonical choice $\sigma\equiv 1$ yields the boundary condition used in \eqref{Helm_PDE}.
 The impedance boundary conditions cannot be derived from Hamilton's variational principle associated with the real Lagrangian \eqref{eq:hlagrdens}, since real Lagrangians generate conservative systems and therefore do not capture dissipative boundary effects.
}

 {However, this obstruction can be remedied formally by adding a purely imaginary boundary term to the real action
$\mathcal{J}=\int_{\Omega} \ell \d x$
thereby obtaining a complex action. This additional term models the boundary $\partial\Omega$ as an energy-dissipating membrane, see \cite{FR2015}, where boundary conditions are interpreted as physical processes occurring within thin boundary layers.}

 {To proceed, we first make the following observation. We write the Lagrangian $\ell$ in the equivalent form
\begin{equation} \label{eq:mdhlagrdens}
 \ell\Big(x, u, \ubar, \nabla u, \nabla \ubar\Big)
 := \frac12 \Big( \nabla u \cdot \nabla \ubar - k^2 u\ubar \Big)
 - \Re(f\ubar),
\end{equation}
and consider the action $\mathcal{J}=\int_{\Omega} \ell \d x$ as a functional of the complex-valued functions $u$ and $\ubar$ (equivalently, of $\Re u$ and $\Im u$). Following the variational formalism used in the derivation of the Schr\"odinger equation from a real Lagrangian \cite{Mo_book}, we vary $u$ and $\ubar$ independently to derive simultaneously the Helmholtz equation and its complex conjugate.}

 {Then we construct the complex action
\begin{equation}
\mathcal{S}(u\ , \ubar) := \int_{\Omega} \ell\Bigl(x, u, \ubar, \nabla u, \nabla \ubar\Bigr) \d x - \frac{ik}{2}\int_{\partial \Omega}u \ubar \d S \ ,
\end{equation}
and compute the Wirtinger derivative with respect to $\ubar$ in the direction $\vbar$,
\begin{equation}\label{eq:wirtubar}
D_{\ubar}\mathcal{S}(u\ , \ubar) := \frac{d}{d\epsilon}\mathcal{S}(u, \ubar+\epsilon \vbar)\Big|_{\epsilon=0} \ .
\end{equation}
We obtain
\begin{equation}\label{Section2:complexVF}
D_{\ubar}\mathcal{S}(u\ , \ubar)= \frac12 \int_\Omega \Bigl(\nabla u\cdot \nabla \vbar  - k^2 \, u\vbar\Bigr)\, \d x -\frac12 \int_\Omega f \vbar \, \d x
          -\frac{ i\, k}{2}\int_{\partial \Omega} u \vbar\, \d S \, .
\end{equation}}
 {By imposing the stationarity condition $D_{\ubar}\mathcal{S}(u\ , \ubar)=0$ and integrating by parts in the first integral, we recover both the Helmholtz equation for $u$ in $\Omega$ and the impedance boundary condition on $\partial \Omega$. This stationarity condition is equivalent to the weak formulation of the impedance boundary value problem, see \eqref{eq:wf-gen} with $\zeta_\Omega =f$ and  $\eta_{\partial\Omega} =0$ in the next section.
It is important to note that, when we impose the stationarity condition $D_{u}\mathcal{S}(u\ , \ubar)=0$, we recover the Helmholtz equation for $\ubar$ in $\Omega$ together with the conjugate impedance boundary condition
$ {\partial_{\mathbf{n}} \ubar} = -i k \ubar$ on $\partial \Omega$. This distinction arises from the fact that the problem is non-selfadjoint.}

\section{Variational principles for the Helmholtz equation}
\label{sec:variational}

\subsubsection*{Notation and preliminaries}

We introduce some standard notation. {Throughout, $\Omega\subset\mathbb{R}^{\dimx}$ is a bounded Lipschitz domain with outward unit normal $\mathbf{n}$ \cite{Grisvard_book}.} For a domain $D$, we write $(\cdot,\cdot)_D$ for the
$L^2$ inner product on $D$, $\|\cdot\|_D$ for the corresponding 
$L^2$ norm, and $\|\cdot\|_{m,D}$ for the Sobolev norm on $H^m(D)$. When $D=\Omega,$ we omit the subscript $D.$ {On the boundary we use the trace spaces $H^s(\partial\Omega)$, $s\in\{\tfrac32,\tfrac12,-\tfrac12,-\tfrac32\}$, and the trace operator $\tau :H^1(\Omega)\to H^{1/2}(\partial\Omega)$ \cite{Lions_Magenes_I}. When no confusion may arise we will use just $u$ to denote $ \tau u\, .$ }

It will be useful to consider a generalised impedance problem on $\Omega$: seek {a} complex-valued function $w$ such that
\begin{equation}\label{eq:gen-Helm}
- \Delta w - k^2 w = \zeta_\Omega \quad \text{in } \Omega, 
\qquad \partial_{\mathbf{n}} w - i k  w = \eta_{\partial\Omega} \quad \text{on } \partial\Omega,
\end{equation}
where the source terms $\zeta_\Omega$ and $\eta_{\partial\Omega}$ are given. {For the minimal weak setting we take $\zeta_\Omega\in H^{-1}(\Omega)$ and $\eta_{\partial\Omega}\in H^{-1/2}(\partial\Omega)$. In several places below we also work in an $L^2$-based setting, writing $\zeta_\Omega\in L^2(\Omega)$ and $\eta_{\partial\Omega}\in L^2(\partial\Omega)$ when stronger regularity is required by the functionals.}

 {We denote  by $\Hc$ the space  $H^1(\Omega)$ equipped with the wavenumber-dependent norm 
\[
\|v\|_{H^1_k(\Omega)} := \Bigl(\|\nabla v\|_\Omega^2 + k^2 \|v\|_\Omega^2\Bigr)^{1/2}\, .
\]}

{For energies that involve the bulk residual $ \L v:=-\Delta v-k^2 v $ in $L^2(\Omega)$ we use the  space}
\[
\Vc := \Bigl\{  v\in H^1(\Omega) :  \L v \in L^2(\Omega),\ v\in H^1(\partial \Omega),  \text{and}\ \partial_{\mathbf{n}} v - i k  v \in  L^2(\partial \Omega ) 
 \Bigr\},
\]
{endowed with the norm}
\[
\|v\|_{\Vc}^2 := \|\nabla v\|_\Omega^2 + k^2 \|v\|_\Omega^2 + \|\L v\|_\Omega^2  + k ^2 \|v \|_{\partial \Omega} ^2  + \| \nabla v \|_{\partial \Omega} ^2 + \| \partial_{\mathbf{n}} v - i k  v \|_{\partial \Omega} ^2 .
\]
{The homogeneous-impedance subspace is}
\[
\Vcbc := \Bigl\{  v\in \Vc  :  \partial_{\mathbf{n}} v - i k  v = 0 \ \text{in } L^2(\partial\Omega)  \Bigr\},
\]
and in the sequel we shall use the notation
$$
  \normVB{v}^2:=\|\nabla v\|_\Omega^2 + k^2 \|v\|_\Omega^2 + \|\L v\|_\Omega^2 + k ^2 \|v \|_{\partial \Omega} ^2 +  \| \nabla v \|_{\partial \Omega} ^2 .
$$

The regularity and other properties of  \eqref{eq:gen-Helm} can be formulated with the aid of its variational formulation. To this end, consider the sesquilinear form $\Bc$   defined by 
\begin{equation}   \label{eq:bf}
  \Bc(u,v) := \int_\Omega \nabla u\cdot \nabla \vbar \, \d x - k^2 \int_\Omega  u\vbar\, \d x
           - i\, k\int_{\partial \Omega} u \vbar\, \d S \, .
\end{equation}
Given $\zeta_\Omega$ and $\eta_{\partial\Omega}$ as above, the weak problem is:

\emph{Find $w\in \Hc$ such that}
\begin{equation}\label{eq:wf-gen}
\Bc(w,v) = (\zeta_\Omega,v)_\Omega  + {\langle \eta_{\partial\Omega},   v \rangle_{H^{-1/2}, H^{1/2}}}
\qquad \text{for all } v\in \Hc .
\end{equation}
{In the $L^2$-based setting the duality pairing reduces to the $L^2(\partial\Omega)$ inner product.}

{By construction, $\Vc$ is the class of functions for which the data-to-solution relation in \eqref{eq:wf-gen} is meaningful with $\zeta_\Omega\in L^2(\Omega)$ and $\eta_{\partial\Omega}\in L^2(\partial\Omega)$, and $\Vcbc\subset\Vc$ corresponds to the case $\eta_{\partial\Omega}=0$. The solution $u$ of \eqref{Helm_PDE} satisfies \eqref{eq:wf-gen} with $\zeta_\Omega=f$ and $\eta_{\partial\Omega}=0$, hence $u\in\Vcbc$,} see \cite{Moiola_Spence2014} and Remark \ref{Rem3.4}.

\subsubsection*{Variational Principles}
In Section \ref{sec:energy}, we 
show that the physical energy associated with the Helmholtz equation  {in unrestricted domains} is
\begin{equation}\label{Energy_H_main_s2}
\mathcal{E}_P(v) = \int_{\Omega }  \frac12  \Bigl(    |\nabla v |^2 - k^2|v|^2 \Bigr)   \d x  -     \Re \int_{\Omega } f   \overline{v}   \d  x   .
\end{equation}
Then the map  $\mathcal{E}_P  : \Hc \to \R $ is clearly differentiable. 
{We take variations with respect to the \emph{real} inner product on $H^1(\Omega;\mathbb{C})$, i.e. $\langle \phi,\psi\rangle_{\mathbb{R}}:=\Re \int_{\Omega}\phi \overline{\psi} \d x$, so that $D\mathcal{E}_P(u)[v]\in\mathbb{R}$ for all $u,v$.}
To find its derivative we first notice 
\begin{equation}\label{Energy_H_main_Re}
\mathcal{E}_P(v) =  \Re \Bigl\{  \int_{\Omega }  \frac12  \bigl( |\nabla v |^2 - k^2|v|^2 \bigr)  \d x  -  \int_{\Omega } f   \overline{v}   \d  x \Bigr\}. 
\end{equation}
Then, 
\begin{equation}\label{Energy_H_main_Re2}\begin{split}
\< D\mathcal{E}_P(u)  , v  \>   =&  \Re \Bigl\{  \int_{\Omega }  \tfrac12 \Bigl(   \nabla u\cdot \nabla \vbar  
+   \nabla v\cdot \nabla \ubar  - k^2( u \vbar + v \ubar ) \Bigr) \d x  -  \int_{\Omega } f   \overline{v}   \d  x \Bigr\}\\[2pt]
 = &  \Re \Bigl\{  \int_{\Omega }   \bigl( \nabla u\cdot \nabla \vbar  - k^2 u \vbar \bigr)   \d x  -  \int_{\Omega } f   \overline{v}   \d  x \Bigr\}.
\end{split}
\end{equation}
 Consider {  $u= u_R + i u_I $, $v= v_R + i v_I$, and $f = f_R + i f_I$}. 
Then stationary points of $ \mathcal{E}_P, $ i.e., $u\in \Hc $ satisfying,
\begin{equation}\label{Energy_H_main_stat_p}\begin{split}
\< D\mathcal{E}_P(u)  , v  \>  =   0, \quad \text{for all } \ v \in \Hc   , \end{split}
\end{equation}
are such that, for any $v\in \Hc , $
\begin{equation}\label{Energy_H_main_Re_system}
\begin{split}
  &    \int_{\Omega }   \Bigl(     \nabla u _R \cdot \nabla v_R 
 \ - k^2  u_R   v_R \Bigr) \d x  -  \int_{\Omega }{   f_R}    v_R   \d  x   =  0  , \\
 &    \int_{\Omega }   \Bigl(     \nabla u _I \cdot \nabla v_I 
 \ - k^2  u_I   v_I \Bigr) \d x  -   {    \int_{\Omega } f_I    v_I   \d  x }  =  0  . \\ 
\end{split}
\end{equation}
{Since $C_c^\infty(\Omega)\subset \Hc$ is dense, testing \eqref{Energy_H_main_stat_p} with $v\in C_c^\infty(\Omega)$ yields the Euler-Lagrange equation in distributional form.}
Considering test functions $v\in \Hc \cap C ^\infty _ c (\Omega)$ we conclude that stationary points 
of $ \mathcal{E}_P  $  satisfy $- \Delta u - k^2 u = f$ in $ \Omega,$ 
in the sense of distributions.
\begin{proposition}
 Consider the energy $\mathcal{E}_P$ defined in \eqref{Energy_H_main_s2}. The stationary points of 
 $\mathcal{E}_P$ defined on $\Hc$ satisfy the Helmholtz equation in the sense of distributions.
\end{proposition}
{As discussed in Section 2, we cannot incorporate the absorbing boundary conditions into the energy when we restrict the computational domain. One way to view this is to split the domain as $\Omega = \Omega_{\text{Restr}} \cup \Omega_2$. The corresponding energies can then be written as follows, assuming that $f$ has compact support in $\Omega_{\text{Restr}}$:}
\begin{equation}\label{Energy_H_main_s2_2}
\begin{split}
\mathcal{E}_P(v) = &\int_{ \Omega _{\text{Restr}}  }  \frac12  \Bigl(    |\nabla v |^2 - k^2|v|^2 \Bigr)   \d x  -     \Re \int_{\Omega _{\text{Restr}}  } f   \overline{v}   \d  x  +  \int_{\Omega _{2} }  \frac12  \Bigl(    |\nabla v |^2 - k^2|v|^2 \Bigr)   \d x \\
=&: \mathcal{E}_{P, \Omega _{\text{Restr}} }(v) + \mathcal{E}_{P, \Omega _{2}}(v)
\end{split}\end{equation}
An energetically consistent approach requires representing the contribution of the complementary domain, $\mathcal{E}_{P, \Omega _{2}}(v)$,   through a surface integral incorporating appropriate boundary conditions.   As discussed in Section 2, it remains unclear whether impedance conditions can be used to represent this energy contribution. To preserve the variational structure associated with the boundary conditions in (1.1), we therefore leverage complex Lagrangians.

To this end, we  define 
\begin{equation}\label{Energy_H_complex_1}
\begin{split}
\mathcal{E}_{P, \C } (v) = \mathcal{E}_{P, \C } (v,  \overline{v}) =  &\int_{ \Omega    }  \frac12  \Bigl(    |\nabla v |^2 - k^2|v|^2 \Bigr)   \d x  -     \Re \int_{\Omega   } f   \overline{v}   \d  x   - \frac{ik}{2}\int_{\partial \Omega}|v|^2 \d x  \\
 =  &\int_{ \Omega    }  \frac12  \Bigl(    \nabla v\cdot \nabla \vbar  - k^2v \vbar \Bigr)   \d x  -     \Re \int_{\Omega   } f   \overline{v}   \d  x   - \frac{ik}{2}\int_{\partial \Omega}v \vbar \d x  
\end{split}\end{equation}

The calculations leading to \eqref{Section2:complexVF} show the following

\begin{proposition}
 Consider the Lagrangian $\mathcal{E}_{P, \C }$ given by \eqref{Energy_H_complex_1} and defined on $\Hc$.
If $ \Bc(\cdot ,\cdot) $ is the  form defined in \eqref{eq:bf}, the stationary points of $\mathcal{E}_{P, \C }$ with respect to $\vbar,$  
 satisfy, 
 \begin{equation*}   \label{eq:bf_new}
  \Bc(u,v) = \int_{\Omega   } f   \overline{v}   \d  x, \quad \text{for all } v\in \Hc\, , 
\end{equation*}
 i.e., they are weak solutions of    
 Helmholtz problem \eqref{Helm_PDE}.   
\end{proposition}

\begin{remark} [Relation to the approach of Courant and Hilbert]
For the variational treatment of general boundary conditions (for example, real Robin boundary conditions for the Laplacian on a bounded domain) Courant and Hilbert \cite{CH1953}, Sec. 5 (see also \cite{C1943}, Sec. 2), introduced a real boundary Lagrangian density so that the unwanted boundary term arising from integration by parts yields the desired boundary condition. In light of this observation, it is natural to introduce a complex boundary Lagrangian density to derive impedance and other complex boundary conditions.
\end{remark}

\begin{remark} [Other boundary conditions]
It would be interesting to extend the above derivation to other classes of boundary conditions, as well as to alternative approaches for reducing the computational domain, such as perfectly matched layers (PML). Furthermore, alternative methods based on approximating $\mathcal{E}_{P, \Omega _{2}}(v)$ in \eqref{Energy_H_main_s2_2} may also be explored.  
\end{remark}

\subsubsection*{Regularised Lagrangians for Helmholtz}

The Lagrangians derived above provide a variational characterisation of the Helmholtz equation from first principles. However, they are indefinite and lack coercivity. To address this, we construct regularised functionals that penalise the residual of the Helmholtz operator in a least-squares sense. In doing so, we aim to gain a deeper understanding of the analytical properties of the reduced wave equation and, in turn, develop tools for the design of numerical approximations with potentially improved properties.

Specifically, the introduction of the following real Lagrangian will be instrumental 
\begin{equation}
\begin{split}
\mathcal{F}_{\gamma} (v)
&= \int_{\Omega }  \frac12  \Bigl( |\nabla v |^2 - k^2|v|^2 \Bigr)  \d x
 -  {\Re \int_{\Omega } f  \overline{v}  \d  x} \\
&\quad + \gamma_1 \int_{\Omega} \bigl| \L v - f \bigr|^2   \d x
 +  \gamma_2 \int_{\partial \Omega} \Bigl| \partial_{\mathbf{n}} v - i k  v \Bigr|^2   \d S.
\end{split}
\end{equation}
Following the above reasoning, we introduce the corresponding complex Lagrangian to incorporate the influence of the impedance boundary conditions into the principal part
\begin{equation}\label{Complex-Energy_H_gamma_wbc}
\begin{split}
\mathcal{F}_{\gamma,\,  \C \, } (v)
&= \int_{\Omega }  \frac12  \Bigl( |\nabla v |^2 - k^2|v|^2 \Bigr)  \d x
 -  {\Re \int_{\Omega } f  \overline{v}  \d  x} \\
&\quad + \gamma_1 \int_{\Omega} \bigl| \L v - f \bigr|^2   \d x
 +  \gamma_2 \int_{\partial \Omega} \Bigl| \partial_{\mathbf{n}} v - i k  v \Bigr|^2   \d S - \frac{ik}{2}\int_{\partial \Omega}|v|^2 \d x .
\end{split}
\end{equation}
To emphasise the role of $v$ and $\vbar$ as independent variables when differentiating complex Lagrangians,   and using the obvious notation, we denote 
\begin{equation}\label{Complex-Energy_v_vbar}
\begin{split}
 \mathcal{F}_{\gamma,\,  \C \, } (v,  \overline{v}) = \,  &\mathcal{F}_{\gamma,\,  \C \, } (v)\, .
\end{split}\end{equation} 
The next section is devoted to the proof of our key result,  stating that the homogeneous part of   $F_\gamma(u), $ $F_\gamma(u):=\mathcal{F}_\gamma(u)\big|_{f=0}$ is coercive in $\Vc.$ In particular, we 
shall assume 
 the following geometric assumption on $\Omega$, compare to \cite{Moiola_Spence2014};  see Remark \ref{Rem3.6}.

\begin{assumption}\label{ass_omega}
There exists a constant $L_0>0$ such that
$\vec{x}\cdot\vec{n}\ge L_0$ for all $\vec{x}\in\partial\Omega$,
i.e. $\Omega$ is strictly star-shaped with respect to the origin.
Since $L=\mathrm{diam}(\Omega)$, one also has
\[
   L_0 \le \vec{x}\cdot\vec{n} \le L
   \quad \text{for all } \vec{x}\in\partial\Omega .
\]
\end{assumption}

The following coercivity result is the key analytical input. Its proof is
given in Section~\ref{sec:proofs}.

\begin{theorem} %[Coercivity with weakly enforced impedance]
\label{theorem_coerc_wbc}
Assume that Assumption~\ref{ass_omega} holds. Let
$\alpha>\tfrac12$, $\beta>0$ and
$\varepsilon_1,\varepsilon_2,\varepsilon_3>0$. Then, for all $u\in\Vc$,
\begin{equation}\label{coerc_2}
\begin{split}
\dimx F_\gamma(u)
 \ge & \Big(\dimx \gamma_1
 - \frac{\alpha^2 L^2}{\varepsilon_1}
 - \frac{2\beta^2}{(\alpha-\tfrac12)\dimx}\Big) \|\L u\|^2
 + \Big(\frac{\dimx}{2}-\alpha(\dimx-2)-\varepsilon_1\Big)\|\nabla u\|^2 \\
& + \frac12(\alpha-\tfrac12)\dimx k^2\|u\|^2 \\
& + \Big(\alpha L_0-\alpha L^2\varepsilon_3-\alpha\varepsilon_2\Big)
   \|\nabla u\|^2_{\partial\Omega} \\
& + \Big(2\beta-\alpha L-\frac{\alpha L^2}{\varepsilon_2}\Big)
   k^2\|u\|^2_{\partial\Omega} \\
& + \Big(\dimx\gamma_2-\frac{\alpha}{\varepsilon_3}\Big)
   \|\partial_{\mathbf n}u-iku\|^2_{\partial\Omega}.
\end{split}
\end{equation}
Consequently, there exist positive constants
$c_0,c_1,c_2,c_3$ and thresholds $\gamma_{1,0},\gamma_{2,0}$,
all independent of $k$, such that, for
$\gamma_1\ge\gamma_{1,0}$ and $\gamma_2\ge\gamma_{2,0}$,
\begin{equation}
\label{coerc_d}
\begin{split}
F_\gamma(u) \ge
c_0\big(\|\L u\|^2+\|\nabla u\|^2+k^2\|u\|^2\big)
+ c_1\|\nabla u\|^2_{\partial\Omega}
+ c_2 k^2\|u\|^2_{\partial\Omega}
+ c_3\|\partial_{\mathbf n}u-iku\|^2_{\partial\Omega}.
\end{split}
\end{equation}
\end{theorem}

As a consequence of Theorem~\ref{theorem_coerc_wbc}, we have the following result.
\begin{theorem}\label{Theorem_SCVP_2}
Consider the Lagrangian  $\mathcal{F}_{\gamma,\,  \C \, } $ defined in \eqref{Complex-Energy_H_gamma_wbc}. %\eqref{Energy_H_gamma_wbc}. 
Then, under the assumptions of Theorem \ref{theorem_coerc_wbc}, the following hold:
\begin{enumerate}
\item[(i)] {There exists $\boldsymbol{\gamma}_0=(\gamma_{1,0},\gamma_{2,0})$ such that for all $\boldsymbol{\gamma}=(\gamma_1,\gamma_2)$ with $\gamma_j\ge \gamma_{j,0}$, $j=1,2$, the  real part of $\mathcal{F}_{\gamma,\,  \C \, } \big|_{f=0}$   is strongly coercive in the norm $\|v\|_{\Vc}$.}
\item[(ii)] The stationary points of $\mathcal{F}_{\gamma,\,  \C \, } $ with respect to $\vbar, $    restricted to $\Vc, $ i.e., the elements   $u\in \Vc$ satisfying 
$$   D _ {\vbar } \, \mathcal{F}_{\gamma,\,  \C \, } ( u ) =0 \, $$
 are {exactly} the solutions of \eqref{Helm_PDE}.
\end{enumerate}
\end{theorem}

\noindent
Statement (i) is precisely Theorem \ref{theorem_coerc_wbc}.
Assuming this coercivity estimate, statement (ii) follows from the
complex version of the Lax--Milgram theorem \cite[Chapter VII]{Dautray_Lions_v2_1988}. To this end, consider the sesquilinear form $\mathcal{A}_{\gamma,\,  \C \, }$ defined by 
\begin{equation}\label{Aform_reg_wbc}
\begin{split}
\mathcal{A}_{\gamma,\,  \C \, }(u,v)
&:= \int_\Omega \nabla u\cdot \nabla \overline{v}  \d x  -  k^2 \int_\Omega u \overline{v}  \d x \\
&\quad +  2\gamma_1 \int_{\Omega} (\L u)  \overline{\L v}  \d x
 +  2\gamma_2 \int_{\partial\Omega} 
\Bigl(\partial_{\mathbf n}u - i k u \Bigr) 
\overline{\Bigl(\partial_{\mathbf n}v - i k v \Bigr)}  \d S \\
&\quad   -  {ik} \int_{\partial \Omega}u \, \vbar  \d x .
\end{split}
\end{equation}
{This form is linear in its first argument and antilinear in its second, and it arises as the first variation of $\mathcal{F}_{\gamma,\,  \C \, }$ with respect to $\vbar .$}
We conclude that
\begin{equation}\label{Energy_H_gamma_Re_2_wbc}
\begin{split}
\langle D_{\vbar\, }  \mathcal{F}_{\gamma,\,  \C \, }(u), v \rangle
=  \Bigl\{ \mathcal{A}_{\gamma,\,  \C \, }(u,v)  -  \int_{\Omega} f  \overline{\bigl(v + 2\gamma_1 \L v\bigr)}  \d x \Bigr\}.
\end{split}
\end{equation}
Thus, the stationary points of $\mathcal{F}_\gamma$ satisfy
\begin{equation}\label{Energy_H_gamma_Re_3_wbc}
\mathcal{A}_{\gamma,\,  \C \, } (u,v)  =  \int_{\Omega} f  \overline{\bigl(v + 2\gamma_1 \L v\bigr)}  \d x ,
\qquad \text{for all } v \in \Vc .
\end{equation}
Theorem \ref{theorem_coerc_wbc} implies
\begin{equation}%\label{Energy_H_gamma_Re_4_wbc}
\Re \  \frac{1}{2} \mathcal{A}_{\gamma,\,  \C \, }(v,v)
 = \Re \  \mathcal{F}_{\gamma,\,  \C \, } (v)\Big|_{f=0}
 \ge  \tilde c  \|v\|_{\Vc}^{2},
\qquad \text{for all } v\in \Vc .
\end{equation}
{Thus $\mathcal{A}_{\gamma,\,  \C \, }$ is bounded on $\Vc \times\Vc $ and $  \mathcal{A}_{\gamma,\,  \C \, }$ is $\Vc$-elliptic, while the right-hand side $v\mapsto \int_\Omega f \overline{(v+2\gamma \L v)}$ defines a bounded antilinear functional on $\Vc$.}
Therefore, by the Lax-Milgram theorem \cite[Section~1, Chapter~VII]{Dautray_Lions_v2_1988} the variational problem 
\eqref{Energy_H_gamma_Re_3_wbc} has a unique solution. 
To complete the proof, consider $u$ being the unique weak solution of \eqref{Helm_PDE}. Then $u$ solves 
 \begin{equation*}
  \Bc(u,v) = \int_{\Omega   } f   \overline{v}   \d  x, \quad \text{for all } v\in \Hc\, . 
\end{equation*}
Furthermore, since $\Omega $ is a Lipschitz domain, $u \in \Vc, $ see Remark \ref{Rem3.4}    . In addition, 
$\L u =f$ in $L^2 (\Omega) $ and the impedance boundary conditions are satisfied in the $L^2 ( \partial \Omega) $
sense. Since $\Vc \subset \Hc $ we conclude that  $u$ satisfies \eqref{Energy_H_gamma_Re_3_wbc} and it is exactly the unique stationary point of $ \mathcal{F}_{\gamma,\,  \C \, }$ with respect to $\vbar\, .$

\subsubsection*{Related minimisation problems}
Our analysis may lead to alternative approaches for approximating \eqref{Helm_PDE}. In particular, it provides a mathematical foundation for iterative methods that incorporate boundary conditions either as Lagrange multipliers or through fixed-point iterations. To fix ideas, consider the following problem: assume that an approximation of $u, $ denoted by $\tilde w$, is given, and we seek to update it by solving the following problem:
  Find $w \in \Vc$ such that 
\begin{equation}\label{fixed_p_1}
\begin{split}
  \int_\Omega \nabla w\cdot \nabla \overline{v}  \d x  &-  k^2 \int_\Omega w \overline{v}  \d x
  +
  2\gamma_1 \int_{\Omega} (\L w)  \overline{\L v}  \d x
 +  2\gamma_2 \int_{\partial\Omega} 
\Bigl(\partial_{\mathbf n}w - i k w \Bigr) 
\overline{\Bigl(\partial_{\mathbf n}v - i k v \Bigr)}  \d S \\
&\quad  =   {ik} \int_{\partial \Omega}\widetilde w \, \vbar \,   \d S  +  \int_{\Omega} f  \overline{\bigl(v + 2\gamma_1 \L v\bigr)}  \d x ,
\qquad \text{for all } v \in \Vc .
\end{split}
\end{equation}
Consider this problem with all data fixed except for  
$\widetilde w.$ We shall show that it admits a unique solution, and denote by $T$ the solution operator that maps   $\widetilde w \mapsto w :$   
$$ T : \Vc \to \Vc \, , \qquad w = T ( \widetilde w )\, .$$
We 
consider the sesquilinear form $ \mathcal{A}_{\text{wbc}} $ defined by 
\begin{equation}
\begin{split}
\mathcal{A}_{\text{wbc}}(u,v)
&:= \int_\Omega \nabla u\cdot \nabla \overline{v}  \d x  -  k^2 \int_\Omega u \overline{v}  \d x \\
&\quad +  2\gamma_1 \int_{\Omega} (\L u)  \overline{\L v}  \d x
 +  2\gamma_2 \int_{\partial\Omega} 
\Bigl(\partial_{\mathbf n}u - i k u \Bigr) 
\overline{\Bigl(\partial_{\mathbf n}v - i k v \Bigr)}  \d S .
\end{split}
\end{equation}
Then $w$ solves
\begin{equation}\label{fixed_p_2}
\mathcal{A}_{\text{wbc}} (w,v)  = {ik} \int_{\partial \Omega}\widetilde w \, \vbar \,   \d S   + \int_{\Omega} f  \overline{\bigl(v + 2\gamma_1 \L v\bigr)}  \d x    
\qquad \text{for all } v \in \Vc .
\end{equation}
Notice that  $\mathcal{A}_{\text{wbc}} $ {is Hermitian (conjugate-symmetric)}, i.e. $\mathcal{A}_{\text{wbc}} (u,v)=\overline{\mathcal{A}_{\text{wbc}} (v,u)}$.
The next theorem shows that $w$ can be obtained as a minimiser of the real Lagrangian
\begin{equation}\label{Energy_H_gamma_wbc}
\begin{split}
\widetilde {\mathcal{F}}_{\gamma} (v)
&= \int_{\Omega }  \frac12  \Bigl( |\nabla v |^2 - k^2|v|^2 \Bigr)  \d x
 -  {\Re \ {ik} \int_{\partial \Omega}\widetilde w \, \vbar \, \d S} -  {\Re \int_{\Omega } f  \overline{v}  \d  x}\\
&\quad + \gamma_1 \int_{\Omega} \bigl| \L v - f \bigr|^2   \d x
 +  \gamma_2 \int_{\partial \Omega} \Bigl| \partial_{\mathbf{n}} v - i k  v \Bigr|^2   \d S \, .
\end{split}
\end{equation}

\begin{theorem}%\label{Theorem_SCVP_2}
Consider the energy $\widetilde {\mathcal{F}}_{\gamma}$ defined in \eqref{Energy_H_gamma_wbc}, and the minimisation problem
\begin{equation}\label{min_V}
\min_{v \in \Vc} \widetilde {\mathcal{F}}_{\gamma} (v).
\end{equation}
Then, under the assumptions of Theorem \ref{theorem_coerc_wbc}, the following hold:
\begin{enumerate}
\item[(i)] {There exists $\boldsymbol{\gamma}_0=(\gamma_{1,0},\gamma_{2,0})$ such that for all $\boldsymbol{\gamma}=(\gamma_1,\gamma_2)$ with $\gamma_j\ge \gamma_{j,0}$, $j=1,2$, the quadratic part of $ \widetilde {\mathcal{F}}_{\gamma}$ (i.e. with $f=0$ and $\widetilde w=0$) is strongly coercive in the norm $\|v\|_{\Vc}$.}
\item[(ii)] The problem \eqref{min_V} has a unique solution {that coincides with the   solution of} \eqref{fixed_p_2}.
\item[(iii)] Both \eqref{min_V} and \eqref{fixed_p_2} admit a unique solution $w$  that satisfies: 
$$ \| w\| _{\Vc} \leq C _{T} \big ( \| \widetilde w \|  _{\Vc} + \| f\|_{L^2 (  \Omega)} \big )  \, .$$
Furthermore, if in addition $\gamma_1\ge \tilde \gamma_{1,0}\geq \gamma_{1,0}$ for an appropriately chosen $\tilde \gamma_{1,0},$ $T$ is a contraction in $ L^2 (\partial \Omega):$ 
$$  \|T( \widetilde w _ 1) - T(\widetilde w _2) \| _{L^2 (\partial \Omega)} %= \| w_1 - w_2 \| _{L^2 (\partial \Omega)} 
\leq \, \eta \,  \| \widetilde w _ 1 - \widetilde w _2\|  _{L^2 (\partial \Omega)}  \, , \quad \text{for a constant } \ 0< \eta < 1\, .   $$
\end{enumerate}
\end{theorem}

As before (i) follows by Theorem \ref{theorem_coerc_wbc}.    Statements (ii) and (iii) then follow from the {complex} Lax-Milgram theorem. In fact,   the sesquilinear form $ \mathcal{A}_{\text{wbc}} $  arises as the first variation of the \emph{real} Lagrangian $\widetilde {\mathcal{F}}_{\gamma} $.
We conclude that
\begin{equation}\label{Energy_H_gamma_Re_2_wbc}
\begin{split}
\langle D\widetilde {\mathcal{F}}_{\gamma}(u), v \rangle
= \Re \Bigl\{ \mathcal{A}_{\text{wbc}}(u,v)  -   \ {ik} \int_{\partial \Omega}\widetilde w \, \vbar \,  \d x - \int_{\Omega} f  \overline{\bigl(v + 2\gamma_1 \L v\bigr)}  \d x \Bigr\}.
\end{split}
\end{equation}
Thus, the stationary points of $\widetilde {\mathcal{F}}_{\gamma} $ are solutions of \eqref{fixed_p_2}. 
%\begin{equation}\label{Energy_H_gamma_Re_3_wbc}
%\mathcal{A}_{W BC}(u,v)  =  \int_{\Omega} f  \overline{\bigl(v + 2\gamma_1 \L v\bigr)}  \d x ,
%\qquad \text{for all } v \in \Vc .
%\end{equation}
Theorem \ref{theorem_coerc_wbc} implies
\begin{equation}\label{Energy_H_gamma_Re_4_wbc}
\Re  \frac{1}{2} \mathcal{A}_{\text{wbc}}(v,v) 
 \ge  \tilde c  \|v\|_{\Vc}^{2},
\qquad \text{for all } v\in \Vc .
\end{equation}
In addition the right-hand side of  \eqref{fixed_p_2} defines a bounded antilinear functional on $\Vc$. {Therefore, by \cite[Section~1, Chapter~VII]{Dautray_Lions_v2_1988}, the variational problem \eqref{fixed_p_2} admits a unique solution, and statement (ii) follows.} 
The fact that $T$ is a contraction in $L^2(\partial \Omega)$ follows from the observation that the boundary term involving $\widetilde w$ in \eqref{fixed_p_2} is independent of $\gamma_1$ and $\gamma_2$, together with the bound \eqref{coerc_2}, provided that $\beta$ is chosen such that
in $\Big(2\beta - \alpha L - \frac{\alpha L^2}{\varepsilon_2}\Big)$
is sufficiently large, which in turn implies that $\gamma_1$ must be taken sufficiently large. The remainder of statement (iii) is then immediate.

\subsubsection*{A convergent fixed point iteration} Let $\rho >0$ and consider  the following iterative scheme:
  \begin{equation}\label{fixed_point_it_tilde}
  	\widetilde w ^{n+1} = (1-\rho)\widetilde w ^{n} + \rho\,  w^{n+1} = (1-\rho)\widetilde w ^{n} + \rho\,  T (\widetilde w ^{n}) =: G (\widetilde w ^{n} )\, .
  \end{equation}  
  Utilising (iii) of the above theorem we have
  $$  \|G( \widetilde w _ 1) - G(\widetilde w _2) \| _{L^2 (\partial \Omega)} %= \| w_1 - w_2 \| _{L^2 (\partial \Omega)} 
\leq \, \Big [(1-\rho) + \eta \rho \, \Big ]\,  \| \widetilde w _ 1 - \widetilde w _2\|  _{L^2 (\partial \Omega)}  \, .$$
Since $ 0< \eta < 1    $ it follows that $G$ is a contraction for any $\rho ,$ $0< \rho < 1\, .$
Hence 
%  {any additional assumptions?} 
the sequence converges to the fixed point $u$ in  $L^2 (\partial \Omega)\, .$
By the stability bound, 
$$ \| u- T(\widetilde w ^{n}) \| _{\Vc} \leq C   \, k \, \| u - \widetilde w ^{n}  \|  _{L^2 (\partial \Omega)}   \, ,$$
we conclude that $T(\widetilde w ^{n}) \to u$ in $\Vc\, .$

\begin{remark}\label {Rem3.3}  
Notice that it is not possible to approximate the solution of the
Helmholtz equation \eqref{Helm_PDE} by solutions of minimisation
problems of the form \eqref{min_V} when the regularised terms are not
present, i.e., $\gamma_1 = \gamma_2 = 0$. In fact, in this case the
corresponding variational formulation \eqref{fixed_p_2} may not even
be well posed.
\end{remark}

\begin{remark}\label {Rem3.4}  
  Due to the fact that the impedance condition induces increased
  regularity at the boundary, weak solutions of the Helmholtz equation
  \eqref{Helm_PDE} in Lipschitz domains satisfy $u \in \Vc$, provided
  that $f \in L^2(\Omega)$. See Theorem 4.24(ii) in \cite{McLean_book}
  and Proposition 3.2 in \cite{Moiola_Spence2014}.
\end{remark}

\begin{remark}\label {Rem3.5}  
  In the case where $\widetilde w$ is taken to be zero, the solution
  of \eqref{min_V} and \eqref{fixed_p_2} still exists and is also the
  solution of the Helmholtz equation \eqref{Helm_PDE}, provided the
  compatibility condition $ik\, u = \partial _{\mathbf{n}} u = 0$
  holds.	
\end{remark}

\begin{remark}\label {Rem3.6}
The geometric Assumption \ref{ass_omega},
$
\vec{x}\cdot\vec{n}\ge L_0$ for all $ \vec{x}\in\partial\Omega,
$
is used to control from below the boundary terms arising in the proof of the coercivity estimate in Theorem 4.2. Whether the same result, or a weaker version of it, can be established under weaker geometric assumptions remains an open question.
 \end{remark}

\begin{remark}\label{remark_MS}
It is interesting to compare the problem \eqref{Energy_H_gamma_Re_3_wbc} to the sesquilinear framework introduced in \cite{Moiola_Spence2014}. {A simplified version of the sesquilinear form $b(\cdot,\cdot)$, shown there to be strongly coercive under appropriate geometric assumptions, can be written as}
\begin{equation}\label{bform_MS}
\begin{split}
b(u,v)  :=  \int_\Omega \nabla u\cdot \nabla \overline{v}  \d x  +  k^2 \int_\Omega  u \overline{v} \d x
 +  \int_{\Omega} \Bigl(-\mathcal{M}u + \frac{1}{3k^2} \L u \Bigr)  \overline{\L v}  \d x \\
-  \int_{\partial\Omega} \Bigl(
ik u \overline{\mathcal{M}v}
 + 
\Bigl( \vec{x}\cdot \nabla_{\partial\Omega} u - i k \beta  u + \frac{\dimx-1}{2} u \Bigr) 
\overline{\Bigl( \frac{\partial v}{\partial n} \Bigr)}\\
 + 
\vec{x}\cdot\vec{n} \bigl(
k^2 u \overline{v} - \nabla_{\partial\Omega} u \cdot \overline{\nabla_{\partial\Omega} v}
\bigr)
\Bigr)  \d S.
\end{split}
\end{equation}
Here $\mathcal{M}u := \vec{x}\cdot \nabla u - i k \beta  u + \frac{\dimx-1}{2} u$ and $\beta\in\mathbb{R}$ {is a free parameter}. {In} \cite{Moiola_Spence2014} it is shown that if $u$ is the solution of the {Helmholtz} equation then
\[
b(u, v)  =  \int_{\Omega } f   \overline{\Bigl(\mathcal{M}v + \frac{1}{3k^2}  \L v\Bigr)}   \d x,
\quad \text{for all } v \in \Vc.
\]
Notice that the space $V$ used in \cite{Moiola_Spence2014} coincides with our space $\Vc$ although the corresponding norms differ in scale. Also, in our definition of $\| \cdot\|_{\Vc}$ the term $\| \L v\| $ appears instead of  $\| \Delta v\| .$  The motivation for introducing \eqref{bform_MS} comes directly from a Morawetz
identity for the expression, \cite[Section 1.4]{Moiola_Spence2014},
\begin{align}\label{morawetz_MS}
\int_\Omega \overline{ \mathcal{M} v } \L u\, \d x
+
\int_\Omega  \mathcal{M}u  \overline{\L v}\, \d x = M (u, v)\, ,
&
\end{align}
where $M(u, v)$ involves bulk and boundary terms and $u, v$ are any smooth enough functions, see \cite[Lemma 2.4]{Moiola_Spence2014}.
If $u$ is a solution of the Helmholtz equation one has
\begin{align}\label{morawetz_MS_2}
\int_\Omega \overline{ \mathcal{M} v } f\, \d x
+
\int_\Omega  \mathcal{M}u  \overline{\L v}\, \d x = M (u, v)\, .
&
\end{align}
The final definition of $b(\cdot, \cdot)$ arises by using the boundary condition and adding a multiple of 
$$ \int_{\Omega} \Bigl( \frac{1}{ k^2} \L u \Bigr)  \overline{\L v}  \d x  = \int_{\Omega} \Bigl( \frac{1}{ k^2} f \Bigr)  \overline{\L v}  \d x $$ to both sides of the above identity, see \cite[Proof of Proposition 3.2]{Moiola_Spence2014} for details. It is then shown that 
 $b(\cdot, \cdot)$ is coercive with respect to the norm on $V.$ It will become evident in the next section that our design strategy takes a different perspective, both in its motivation and in the way Morawetz identities are incorporated.  
\end{remark}

\section{Proofs for strongly coercive variational principles}
\label{sec:proofs}

\subsection*{Design strategy}
\label{sec:design}
Our construction of the enhanced Lagrangians  is intrinsically connected to the analysis establishing their coercivity. 
The key idea is to observe that the energy $ \mathcal{E}_P (u)$ contains the indefinite term   
\begin{equation}\label{energy_wd}
    \frac{1}{2} \int_\Omega \qp{\norm{\nabla u}^2 - k^2 \norm{u}^2} \,  \d x \, , 
\end{equation}
and at the same time a standard Rellich identity, see \eqref{eq:IntRellich}, can be written as
 \begin{equation}
      \label{eq:IntRellich_wd}
    \int_\Omega \qp{(\dimx  -2) \norm{\nabla u}^2 - \dimx k^2 \norm{u}^2} \,  \d x \\
     + \int_\Omega 2 \Re\qp{\qp{\vec{x} \cdot \nabla u} \overline{\L  u}} \,  \d x
    + B_R(u)=0,
 	 \end{equation}
where $ B_R(u)$ accounts for the boundary terms in \eqref{eq:IntRellich}. 
Then, an appropriate linear combination of the principal part of the Lagrangian for $f=0$  and {  the identity} \eqref{eq:IntRellich_wd} 
has two {  key advantages}: the energy {  changes} only by a constant factor, since the contribution of \eqref{eq:IntRellich_wd}  is zero,
and the coefficients  of $\norm{\nabla u}^2$ and $k^2 \norm{u}^2 $ {  can both be made} constants of the same sign. 
The least-squares regularisation term in the Lagrangian  functional {  makes it possible} to control the second term in \eqref{eq:IntRellich_wd}. 
{  Moreover,} the boundary term in \eqref{eq:IntRellich_wd} can also be controlled, {  particularly} through the use of {  another} 
lower-order Morawetz identity---a fact that is not immediately apparent. It is important to emphasise that, within this framework, the structure 
of the regularised Largangian is preserved, and coercivity is ensured for all $\gamma \geq \gamma_0$, where the constant $ \gamma_0$
can be determined explicitly.  

 {Next, to simplify the presentation of our approach and the proof of Theorem 4.2, we first establish the coercivity bound under the assumption that the impedance condition is satisfied exactly; see Theorem 4.1. This is done for all elements of the space $\Vcbc \subset \Vc$.}

\subsubsection*{Rellich and Morawetz Identities}

  Let $\L u = -\Delta u - k^2 u$ be the Helmholtz operator applied to a
  sufficiently smooth function $u \colon \mathbb{R}^\dimx \to
  \mathbb{C}$. Then the following identity holds:
  \begin{align}\label{rellich_id}
    &\qp{\vec{x} \cdot \overline{\nabla u}} \L  u
    +
    \qp{\vec{x} \cdot \nabla u} \overline{\L  u} \nonumber\\
    &=
       -\div\qp{
      (\vec{x} \cdot \overline{\nabla u}) \nabla u
      +
      (\vec{x} \cdot \nabla u) \overline{\nabla u}
      -
      \vec{x} \norm{\nabla u}^2
      +
      k^2 \vec{x} \norm{u}^2
    }
    -
    (\dimx -2) \norm{\nabla u}^2
    +
    \dimx k^2 \norm{u}^2 ,
  \end{align}
where $\vec{x} \cdot \nabla u = \sum_{{  j=1}}^{\dimx} x_j \frac{\partial u}{\partial x_j}$ is the directional derivative {  of $u$} along $\vec{x}$.

Although this identity is well known \cite{CummingsFeng2006, Spence_Chandler-Wilde_integral_2011, Moiola_Spence2014} as the simplest form of a Morawetz {  (or Rellich)} identity, we briefly outline the main {  steps of its derivation} to highlight the role of the Morawetz multiplier.

To begin, {  let  $\Phi$ denote the left-hand side of identity \eqref{rellich_id}. Using the definition of the Helmholtz operator $\L$, we have}
  \begin{equation}
    \Phi
    =
    \qp{\vec{x} \cdot \overline{\nabla u}}
    \qp{-\Delta u - k^2 u}
    +
    \qp{\vec{x} \cdot \nabla u}
    \qp{-\Delta \overline{u} - k^2 \overline{u}}.
  \end{equation}
  Expanding each term {  gives}
  \begin{equation}
    \Phi
    =
    -\qp{\vec{x} \cdot \overline{\nabla u}} \Delta u
    -
    k^2 \qp{\vec{x} \cdot \overline{\nabla u}} u 
    -
    \qp{\vec{x} \cdot \nabla u} \Delta \overline{u}
    -
    k^2 \qp{\vec{x} \cdot \nabla u} \overline{u}.
\end{equation}
The terms involving the Laplacian {  expand as follows:}
\begin{equation}
  -\qp{\vec{x} \cdot \overline{\nabla u}} \Delta u
  -
  \qp{\vec{x} \cdot \nabla u} \Delta \overline{u}
  = 
  -\div\qp{\qp{\vec{x} \cdot \overline{\nabla u}} \nabla u
    +
    \qp{\vec{x} \cdot \nabla u} \overline{\nabla u}
    {  -
     \vec{x} \norm{\nabla u}^2}}
  {  -
  (\dimx - 2)} \norm{\nabla u}^2,
\end{equation}
{  where we have used that}
\begin{equation}
  \div\qp{\vec{x} \norm{\nabla u}^2} = {  \dimx}\, \norm{\nabla u}^2
  + \qp{\vec{x} \cdot \nabla}\qp{\norm{\nabla u}^2}.
\end{equation}
For the {  remaining}  two terms, we {  observe}  that {  they can be written in the form}
\begin{equation}
  - k^2 \qp{\qp{\vec{x} \cdot \overline{\nabla u}} u
    +
    \qp{\vec{x} \cdot \nabla u} \overline{u}}
  =
  - k^2 \qp{\vec{x} \cdot \nabla \qp{\norm{u}^2}}.
\end{equation}
Rewriting {  this} in divergence form, {  we obtain}
\begin{equation}
  - k^2 \qp{\vec{x} \cdot \nabla \qp{\norm{u}^2}}
  =
  -\div \qp{k^2 \vec{x} \norm{u}^2} + \dimx k^2 \norm{u}^2.
\end{equation}
Combining all   the  contributions  derived above, we  recover identity  \eqref{rellich_id}. 
 Integration  over $\Omega$,   together with  the divergence theorem and a density argument  
(see, e.g.,  \cite{CummingsFeng2006, Spence_Chandler-Wilde_integral_2011,  Moiola_Spence2014})
 yields the integrated form of identity \eqref{rellich_id}, stated in the following proposition.  

\begin{proposition}[Rellich/Morawetz Identity]
  \label{cor:intmorawetz}
  {  Let $u\in \Vc,$ $\L u = -\Delta u - k^2 u$   and $\Omega$ be a   domain in $\mathbb{R}^\dimx $ 
  with Lipschitz boundary $\partial \Omega$.} Then the following  identity holds:
  \begin{multline}    \label{eq:IntRellich}
    \int_\Omega 2 \Re\qp{\qp{\vec{x} \cdot \nabla u} \overline{\L  u}} \,  \d x
    +
    \int_\Omega \qp{(\dimx  -2) \norm{\nabla u}^2 - \dimx k^2 \norm{u}^2} \,  \d x \\
    +
    \int_{\partial \Omega} \qp{
      2 \Re\qp{\qp{\vec{x} \cdot \overline{\nabla u}} \qp{\nabla u \cdot \vec{n}}}
      -
      \norm{\nabla u}^2 \qp{\vec{x} \cdot \vec{n}}
      +
      k^2 \norm{u}^2 \qp{\vec{x} \cdot \vec{n}}
    } \,  \d S =0,
  \end{multline}
  where $\vec{n}$ is the outward unit normal vector on $\partial
  \Omega$, and $ \d S$ is the surface element.
\end{proposition}

The following identity will be important for the analysis that follows.

\begin{proposition}[Low-order Morawetz Identity]
  \label{the:boundarymorawetz}
  Let $u\in \Vc$, $\L u = -\Delta u - k^2 u$, and let $\Omega$ be {a bounded} domain in $\mathbb{R}^\dimx$ 
  with Lipschitz boundary $\partial \Omega$. 
  Define
  \begin{equation}  \label{eq:multM0}
    M_0 u :=  - i k \beta  u,
  \end{equation}
  where $\beta \in \mathbb{R}$ is a constant. Then the following integral identity holds:
 \begin{equation} \label {M0}
    \int_\Omega  2  \Re  \big( \overline{M_0 u}  \L  u\big)    \d x
     -    2 k ^2 \beta  \int_{\partial \Omega}   |u|^2    \d S       
     +  2  \Re   \int_{\partial \Omega} i k \beta   \overline{u}   \Big ( \nabla u \cdot \vec{n}
          -ik   u \Big )    \d S  =  0.
  \end{equation}
  {When traces are not $L^2$-valued, the boundary integrals are interpreted in the $H^{-1/2}(\partial\Omega)$-$H^{1/2}(\partial\Omega)$ duality sense.}
\end{proposition}

\begin{proof}
Using the definition of $\L$ we observe
\begin{align*} 
    \int_\Omega \overline{M_0 u}  \L  u    \d x
    &=    \int_\Omega i k \beta   \overline{u}  \L  u    \d x
      =  -   \int_\Omega i k \beta   \overline{u}  \Delta u    \d x  -   \int_\Omega i k^3 \beta   |u|^2    \d x
      \\
    &=  \int_\Omega i k \beta   \overline{\nabla  u} \cdot \nabla u    \d x  -   \int_{\partial \Omega} i k \beta   \overline{u}  \nabla u \cdot \vec{n}   \d S
        -   \int_\Omega i k^3 \beta   |u|^2    \d x
        \\
   &=  \int_\Omega i k \beta   \|\nabla u\|^2   \d x 
        -   \int_\Omega i k^3 \beta   |u|^2    \d x  -   \int_{\partial \Omega} i k \beta   \overline{u}   \Big ( \nabla u \cdot \vec{n}
          -ik   u \Big )    \d S\\
   &\qquad\qquad\qquad\qquad\qquad\qquad +  k ^2 \beta  \int_{\partial \Omega} |u|^2    \d S  .
\end{align*}
Since 
\[
\overline{M_0 u}  \L  u + M_0 u  \overline{\L  u}= 2  \Re  \big( \overline{M_0 u}  \L  u\big),
\]
we conclude
\begin{equation*}
    \int_\Omega  2  \Re  \big( \overline{M_0 u}  \L  u\big)    \d x
     =     2 k ^2 \beta  \int_{\partial \Omega} |u|^2    \d S
         -  2  \Re   \int_{\partial \Omega} i k \beta   \overline{u}  \Big ( \nabla u \cdot \vec{n}
          -ik   u \Big )    \d S ,
\end{equation*}
as required.
\end{proof}

\begin{remark}
{The multiplier $M_0$ in \eqref{eq:multM0} does not include the transport term $\vec{x} \cdot \nabla u$ that appears in standard Morawetz multipliers; it is a lower-order choice tailored to control boundary contributions together with the impedance residual.}
\end{remark}

\subsection*{Lagrangians and coercivity bounds}
{Let $\Omega \subset \mathbb{R}^\dimx$ be a bounded Lipschitz domain (for the Rellich identities we take the multiplier $x$, which may be replaced by $x-x_0$ to avoid assuming the origin lies in $\Omega$).} Consider the Helmholtz equation with a source term:
\[
  \L u = f, \qquad \text{where } \L u := -\Delta u - k^2 u,
\]
subject to the impedance boundary condition
\begin{equation}\label{bc_section4}
  \nabla u \cdot \vec{n} = i k  u \quad \text{on } \partial \Omega.
\end{equation}
 {We shall first examine the coercivity of  }
\begin{equation}\label{Egamma}
  E_\gamma(u)
  :=
  \frac{1}{2} \int_\Omega \qp{\norm{\nabla u}^2 - k^2 \norm{u}^2}    \d x
   + 
  \gamma \int_\Omega \norm{\L u}^2    \d x,
\end{equation}
in the space $\Vcbc \, .$
%{i.e. $E_\gamma(u)=\mathcal{E}_\gamma(u)\big|_{f=0}$.}

\begin{lemma}[Energy Identity with Penalty]
  Let $u \colon \mathbb{R}^\dimx \to \mathbb{C}$ be a sufficiently smooth
  function. 
  Then for any $\alpha \in \mathbb{R}$ the following identity
  holds:
  \begin{equation} \label{dimxEgamma}
      \dimx  E_\gamma(u)  =  E_{\gamma,\Omega}(\dimx,\alpha;u)  +  E_{\gamma,\partial\Omega}(\dimx,\alpha;u),
  \end{equation}
  where
  \begin{align}  
     E_{\gamma,\Omega}(\dimx,\alpha;u)
     &:= \dimx \gamma \Norm{\L u}^2
    +
    \qp{\frac{\dimx}{2} - \alpha(\dimx-2)}\Norm{\nabla u}^2
    +
    \qp{\alpha - \frac{1}{2}} \dimx  k^2 \Norm{u}^2 \nonumber\\
    &\hspace{2.4em}
    - 
    \alpha \int_\Omega 2 \Re\qp{\qp{\vec{x} \cdot \nabla u}  \overline{\L u}}    \d x , \label{dimxEgammabulk} \\
    E_{\gamma,\partial\Omega}(\dimx,\alpha;u)
    &:= 
    - 
    \alpha \int_{\partial \Omega} \qp{
      2 \Re\qp{\qp{\vec{x} \cdot \overline{\nabla u}} \qp{\nabla u \cdot \vec{n}}}
      -
      \norm{\nabla u}^2 \qp{\vec{x} \cdot \vec{n}}
      +
      k^2 \norm{u}^2 \qp{\vec{x} \cdot \vec{n}}
    }    \d S .  \label{dimxEgammabdry} 
 \end{align}
\end{lemma}

\begin{proof}
  Beginning with the definition of $E_\gamma(u)$, we {multiply both sides by $\dimx$} and add zero in the form of the identity \eqref{eq:IntRellich}, 
  scaled by $-\alpha$, to obtain
  \begin{multline*}
    \dimx  E_\gamma(u)
    =
    \dimx  \gamma \Norm{\L u}^2
    +
    \qp{\frac{\dimx}{2} - \alpha(\dimx-2)} \Norm{\nabla u}^2
    +
    \qp{\alpha - \frac{1}{2}} \dimx  k^2 \Norm{u}^2 
    -
    \alpha \int_\Omega 2 \Re\qp{\qp{\vec{x} \cdot \nabla u}  \overline{\L u}}    \d x
    \\
    -
    \alpha \int_{\partial \Omega} \qp{
      2 \Re\qp{\qp{\vec{x} \cdot \overline{\nabla u}} \qp{\nabla u \cdot \vec{n}}}
      -
      \norm{\nabla u}^2 \qp{\vec{x} \cdot \vec{n}}
      +
      k^2 \norm{u}^2 \qp{\vec{x} \cdot \vec{n}}
    }    \d S,
  \end{multline*}
  which is precisely \eqref{dimxEgamma}-\eqref{dimxEgammabdry}.
\end{proof}

\begin{remark}\label{choice_of_alpha}[Ensuring coercivity of the low-order bulk terms in \eqref{dimxEgammabulk}]
  Observe that $\alpha\in\mathbb{R}$ is a free parameter in the energy identity. By choosing $\alpha$ appropriately, we can ensure that the coefficients of the gradient and zeroth-order terms are strictly positive. {For $\dimx=1$ the gradient coefficient is $\tfrac12+\alpha$; for $\dimx=2$ it equals $1$. In both cases, choosing $\alpha>\tfrac12$ also makes the $k^2\|u\|^2$ coefficient $\dimx(\alpha-\tfrac12)$ positive. For $\dimx\ge 3$, taking}
  \[
    {\frac 12<\alpha<\frac{\dimx}{2(\dimx-2)}}
  \]
  {ensures both coefficients are positive.}
\end{remark}
{To control the mixed volume term in \eqref{dimxEgammabulk} we use the following simple estimate.}

\begin{lemma}[Control of Weighted Gradient Norm]
  \label{lem:weightedgrad}
  {Let $\Omega \subset \mathbb{R}^{\dimx}$ be a bounded domain and let $L=\mathrm{diam}(\Omega)$. Then, for any sufficiently smooth function $u\colon\Omega\to\mathbb{C}$ and any fixed $x_0\in\overline{\Omega}$,}
  \begin{equation}
    {\int_\Omega \big|(\vec{x}-x_0)\cdot \nabla u\big|^2   \d x  \le  L^2 \int_\Omega |\nabla u|^2   \d x.}
  \end{equation}
\end{lemma}

\begin{proof}
  {By definition of the diameter, $\sup_{\vec{x}\in\Omega}\|\vec{x}-x_0\|\le L$ for any $x_0\in\overline{\Omega}$. Hence}
  \[
    {|(\vec{x}-x_0)\cdot \nabla u| \le \|\vec{x}-x_0\| \|\nabla u\| \le L \|\nabla u\|,}
  \]
  {and the stated inequality follows upon squaring and integrating over $\Omega$.}
\end{proof}

\begin{proposition}[Coercivity of Bulk Terms]
  Let $\Omega \subset \mathbb{R}^\dimx$ be a bounded domain such that $0 \in \Omega$, and let $L=\mathrm{diam}(\Omega)$ denote the diameter of $\Omega$. For any $\alpha \in \mathbb{R}$ and $\varepsilon_1>0$, the bulk terms in the energy identity \eqref{dimxEgamma} can be bounded from below as
  \begin{multline}
   E_{\gamma,\Omega}(\dimx,\alpha;u) 
    \ge
    \qp{\dimx \gamma - \frac{\alpha^2 L^2}{\varepsilon_1}} \Norm{\L u}^2
    +    
    \qp{\frac{\dimx}{2} - \alpha (\dimx -2) - \varepsilon_1}\Norm{\nabla u}^2
    +
    {\qp{\alpha - \frac{1}{2}} \dimx k^2} \Norm{u}^2.
  \end{multline}
\end{proposition}

\begin{proof}
  We bound the last term in the bulk energy,
  \[
      E_{\gamma,\Omega}(\dimx,\alpha;u)
    =
    \dimx \gamma \Norm{\L u}^2
    +
    \qp{\frac{\dimx}{2} - \alpha (\dimx -2)} \Norm{\nabla u}^2
    +
    \qp{\alpha - \frac{1}{2}}\dimx k^2 \Norm{u}^2
     - 
    \alpha \int_\Omega 2 \Re\qp{(\vec{x}\cdot\nabla u) \overline{\L u}}  \d x.
  \]
  By Young’s inequality, for any $\varepsilon_2>0$,
  \[
    \bigg|\int_\Omega 2 \Re\qp{(\vec{x}\cdot\nabla u) \overline{\L u}}  \d x\bigg|
    \le
    \varepsilon_2 \int_\Omega  \norm{\vec{x}\cdot\nabla u}^2  \d x
    + \frac{1}{\varepsilon_2}\int_\Omega  \norm{\L u}^2  \d x.
  \]
  Using Lemma~\ref{lem:weightedgrad} and choosing {$\varepsilon_2=\varepsilon_1/L^2$}, we obtain
  \[
    \bigg|\int_\Omega 2 \Re\qp{(\vec{x}\cdot\nabla u) \overline{\L u}}  \d x\bigg|
    \le
    \varepsilon_1 \Norm{\nabla u}^2 + \frac{L^2}{\varepsilon_1}\Norm{\L u}^2.
  \]
  Hence
  \[
    - \alpha \int_\Omega 2 \Re\qp{(\vec{x}\cdot\nabla u) \overline{\L u}}  \d x
    \ge
    - |\alpha|\Big(\varepsilon_1 \Norm{\nabla u}^2 + \frac{L^2}{\varepsilon_1}\Norm{\L u}^2\Big)
     \ge 
    - \varepsilon_1 \Norm{\nabla u}^2 - \frac{\alpha^2 L^2}{\varepsilon_1}\Norm{\L u}^2,
  \]
  which, substituted into the expression for $E_{\gamma,\Omega}$, yields the stated bound.
\end{proof}

Now consider the boundary energy term \eqref{dimxEgammabdry}, which we rewrite as
\begin{align}\label{dimxEgamma_boundary}
    E_{\gamma,\partial\Omega}(\dimx,\alpha;u) 
    =&  \alpha \int_{\partial \Omega} \Big(
      \norm{\nabla u}^2  (\vec{x}\cdot\vec{n}) 
      -  k^2 \norm{u}^2  (\vec{x}\cdot\vec{n})
      \Big) \d S  -
    \alpha \int_{\partial \Omega}
      2 \Re\qp{(\vec{x}\cdot\overline{\nabla u}) (\nabla u\cdot\vec{n})}  \d S.
 \end{align}
Next, utilising Assumption \ref{ass_omega}, we estimate $E_{\gamma,\partial \Omega}$ as follows:
\begin{align}\label{dimxEgamma_boundary2}
    E_{\gamma,\partial \Omega}(\dimx,\alpha;u) 
    \ge&  \alpha L_0 \int_{\partial \Omega} |\nabla u|^2  \d S
    -  {\alpha}  L k^2 \int_{\partial \Omega} |u|^2  \d S
    -
    \alpha \int_{\partial \Omega}
      2 \Re \big((\vec{x}\cdot \overline{\nabla u}) (\nabla u\cdot \vec{n})\big)  \d S .
 \end{align}

{First consider the case when $u$ exactly satisfies the boundary condition \eqref{bc_section4}.}
With the notation of Proposition~\ref{the:boundarymorawetz} we have
\begin{equation}\label{eq:M0_identity_strong}
  - \int_\Omega 2 \Re \big(\overline{M_0 u} \L u\big) \d x
   +  2k^2\beta \int_{\partial\Omega} |u|^2 \d S  = 0,
\end{equation}
where $M_0 u := -ik\beta u$.
Adding the null quantity \eqref{eq:M0_identity_strong} to the right-hand side of \eqref{dimxEgamma_boundary2} and using \eqref{bc_section4} gives
\begin{align}\label{dimxEgamma_boundary3}
E_{\gamma,\partial \Omega}(\dimx,\alpha;u) 
    \ge & \alpha L_0 \int_{\partial \Omega} |\nabla u|^2  \d S
     +  \bigl(2\beta - L\bigr) k^2 \int_{\partial \Omega} |u|^2  \d S \nonumber\\
 & - 
    \alpha \int_{\partial \Omega}
      2 \Re \big((\vec{x}\cdot \overline{\nabla u}) (ik u)\big)  \d S
     -  \int_\Omega 2 \Re \big(\overline{M_0 u} \L u\big) \d x .
\end{align}
Using Assumption~\ref{ass_omega} and Young’s inequality,
\begin{equation}\label{dimxEgamma_boundary_inder_1}
\begin{aligned}
     \alpha \int_{\partial \Omega}
      2 \Re \big((\vec{x}\cdot \overline{\nabla u}) (ik u)\big)  \d S
      &\le \alpha \int_{\partial \Omega} 2L  |\nabla u|  k|u|  \d S
      \leq
      \frac{\alpha L_0}{2}\int_{\partial \Omega} |\nabla u|^2  \d S
       +  \frac{2\alpha L^2}{L_0} k^2 \int_{\partial \Omega} |u|^2  \d S .
\end{aligned}
\end{equation}
Furthermore, by Young’s inequality (for any $\varepsilon_3>0$),
\begin{equation}\label{dimxEgamma_boundary_inder_2}
\begin{aligned}
     \int_\Omega 2 \Re \big(\overline{M_0 u} \L u\big) \d x
      &\le \int_\Omega 2 \beta k  |u|  |\L u|  \d x 
      \le \varepsilon_3  k^2 \int_{\Omega} |u|^2  \d x
       +  \frac{\beta^2}{\varepsilon_3}\int_{\Omega} |\L u|^2  \d x .
\end{aligned}
\end{equation}

{Combining \eqref{dimxEgamma_boundary3}-\eqref{dimxEgamma_boundary_inder_2} with the bulk estimate and an appropriate choice of parameters $(\alpha,\beta,\varepsilon_1,\varepsilon_3)$ yields the boundary control needed for the global coercivity. This leads to the theorem stated next.}

\begin{theorem}\label{Th_coerc_gamma}{[Coercivity of the energy $E_{\gamma}$]}
Let $\Omega \subset \mathbb{R}^{{\dimx}}$ be a bounded domain such that $0 \in \Omega$, and let $L=\mathrm{diam}(\Omega)$. Assume that Assumption~\ref{ass_omega} holds. Then, for any $\alpha>\tfrac12$, any $\beta>0$, and any $\varepsilon_1>0$, the following estimate holds for all sufficiently smooth $u$ satisfying the boundary condition \eqref{bc_section4}:
\begin{equation}
\label{coerc_1}
\begin{split}
{\dimx} E_{\gamma}(u) 
 \ge  & 
\Bigg({\dimx} \gamma  -  \frac{\alpha^2 L^2}{\varepsilon_1}  -  \frac{2 \beta^2}{{\big(\alpha-\tfrac12\big) \dimx}} \Bigg)  \|\L u\|^2
 +  \Big(\tfrac{{\dimx}}{2} - \alpha({\dimx}-2) - \varepsilon_1\Big) \|\nabla u\|^2
 +  \tfrac12 \big(\alpha-\tfrac12\big) {\dimx} k^2 \|u\|^2
\\
& +  \frac{\alpha L_0}{2} \|\nabla u\|_{\partial\Omega}^2
 +  \Big(2\beta - \alpha L - \frac{2\alpha L^2}{L_0}\Big) k^2 \|u\|_{\partial\Omega}^2.
\end{split}
\end{equation}
Furthermore, there exist positive constants $c_0,c_1$ and $\gamma_0$, {independent of $k$}, such that for all $\gamma\ge \gamma_0$,
\begin{equation}
\label{coerc_c}
\begin{split}
E_{\gamma}(u) 
 \ge  & 
c_0 \Big( \|\L u\|^2 + \|\nabla u\|^2 + k^2 \|u\|^2 \Big)
 + 
c_1 \Big( \int_{\partial \Omega} |\nabla u|^2 \d S  +  k^2 \int_{\partial \Omega} |u|^2 \d S \Big).
\end{split}
\end{equation}
\end{theorem}

\subsubsection*{Proof of Theorem 3.1}

{We now consider the case where $u$ satisfies the impedance condition \eqref{bc_section4} \emph{weakly}} and examine the coercivity of $F_\gamma(u):=\mathcal{F}_\gamma(u)\big|_{f=0}= \, \Re \, \mathcal{F}_{\gamma,\,  \C \, } (v) \big|_{f=0}$.

We have
\begin{equation}\label{dimxFgamma}
      \dimx  F_\gamma(u)  =  F_{\gamma_1}(u)  +  F_{\gamma_2}(u),
\end{equation}
where $ F_{\gamma_1}(u) :=
E_{\gamma_1,\Omega}(u)+E_{\gamma_1,\partial\Omega}(u)$ and
$F_{\gamma_2}(u) := \dimx
\gamma_2\int_{\partial\Omega}\Bigl|\partial_{\mathbf n}u-ik u\Bigr|^2
\d S.$ {The bulk estimate for $E_{\gamma_1,\Omega}$ is identical to
  that for $E_{\gamma,\Omega}$ from the strongly imposed case with
  $\gamma$ replaced by $\gamma_1$.}
 
 For the boundary term
$E_{\gamma_1,\partial\Omega}$ we start from
\eqref{dimxEgamma_boundary2} (with $\gamma=\gamma_1$) and split the
normal derivative:
\begin{align*}
-\alpha\int_{\partial\Omega}2 \Re \big((\vec x \cdot \overline{\nabla u})(\nabla u \cdot \vec n)\big) \d S
&= -\alpha\int_{\partial\Omega}2 \Re \big((\vec x \cdot \overline{\nabla u}) (\partial_{\mathbf n}u-ik u)\big) \d S 
-\alpha\int_{\partial\Omega}2 \Re \big((\vec x \cdot \overline{\nabla u}) ik u\big) \d S.
\end{align*}
By Cauchy-Schwarz/Young's inequality, $| \vec x \cdot \nabla u |\le L|\nabla u|$, so for any $\varepsilon_3>0$,
\[
-\alpha\int_{\partial\Omega}2 \Re \big((\vec x \cdot \overline{\nabla u}) (\partial_{\mathbf n}u-ik u)\big) \d S
 \ge 
-\alpha L^2\varepsilon_3 \|\nabla u\|^2_{\partial\Omega} - \frac{\alpha}{\varepsilon_3} \|\partial_{\mathbf n}u-ik u\|^2_{\partial\Omega}.
\]
Likewise, for any $\varepsilon_2>0$,
\[
-\alpha\int_{\partial\Omega}2 \Re \big((\vec x \cdot \overline{\nabla u}) ik u\big) \d S
 \ge 
-\alpha\varepsilon_2 \|\nabla u\|^2_{\partial\Omega} - \frac{\alpha L^2}{\varepsilon_2} k^2\|u\|^2_{\partial\Omega}.
\]
{To handle the bulk coupling that appears in the strong case via $M_0$, we estimate directly by Young’s inequality: for any $\beta>0$ and any $\varepsilon_4>0$,}
\[
\int_\Omega 2 \Re \big(\overline{(-ik\beta u)} \L u\big) \d x
 \le  \varepsilon_4 k^2\|u\|^2  +  \frac{\beta^2}{\varepsilon_4} \|\L u\|^2,
\]
{which yields the same volume contributions as in the strong case after redistribution into the bulk bound.}

Collecting the above bounds with \eqref{dimxEgamma_boundary2} gives
\begin{equation}\label{coerc_1_wbc}
\begin{split}
E_{\gamma_1,\partial\Omega}(u)
 \ge & \alpha L_0 \|\nabla u\|^2_{\partial\Omega}
 +  (2\beta-\alpha L) k^2\|u\|^2_{\partial\Omega}
 - \alpha L^2\varepsilon_3 \|\nabla u\|^2_{\partial\Omega}
 - \frac{\alpha}{\varepsilon_3} \|\partial_{\mathbf n}u-ik u\|^2_{\partial\Omega}\\
&- \alpha\varepsilon_2 \|\nabla u\|^2_{\partial\Omega}
 - \frac{\alpha L^2}{\varepsilon_2} k^2\|u\|^2_{\partial\Omega}
 - \frac{(\alpha-\tfrac12) \dimx}{2} k^2\|u\|^2
 - \frac{2\beta^2}{(\alpha-\tfrac12) \dimx} \|\L u\|^2.
\end{split}
\end{equation}
Adding $F_{\gamma_2}(u)$ from \eqref{dimxFgamma} yields the net
boundary-residual coefficient ${\big(\dimx \gamma_2 -
  \tfrac{\alpha}{\varepsilon_3}\big) \|\partial_{\mathbf n}u-ik
  u\|^2_{\partial\Omega}.}$ Combining the bulk estimate (with $\gamma
= \gamma_1$) and \eqref{coerc_1_wbc} we obtain the result of Theorem
\ref{theorem_coerc_wbc}.

{One admissible parameter choice is obtained by taking $\alpha\in(\tfrac12,\frac{\dimx}{2(\dimx-2)})$ (or any $\alpha>\tfrac12$ if $\dimx\in\{1,2\}$), then fixing small $\varepsilon_1,\varepsilon_2,\varepsilon_3>0$ so that the boundary coefficients are positive, choosing $\beta$ to make $2\beta-\alpha L - \frac{\alpha L^2}{\varepsilon_2}>0$, and finally taking $\gamma_1,\gamma_2$ above the indicated thresholds.}

\begin{remark}[Explicit calculation of the constants]
 A {  convenient} choice of the constants {  appearing} in front of the boundary norms is  
\[
   \varepsilon_2=\frac{L_0}{4}, \qquad \varepsilon_3=\frac{L_0}{4L^2}.
   \]
   With these values one obtains $\alpha L_0 - \alpha L^2
     \varepsilon_3 - \alpha \varepsilon_2 = \frac{\alpha L_0}{2}, $ $2
     \beta - \alpha L - \frac {\alpha L^2 } {\varepsilon_2} = 2 \beta
     - \alpha L - \frac{4 \alpha L^2}{L_0},$ $\dimx
     \gamma_2-\frac{\alpha}{\varepsilon_3} = \dimx \gamma_2- \frac{4
       \alpha L^2}{L_0} .$ { Coercivity requires these three
       quantities to be positive; in particular
\[
   \beta > \frac{\alpha L}{2} + \frac{2\alpha L^2}{L_0}, \qquad \gamma_2  >  \frac{4 \alpha L^2}{\dimx L_0}.
\]
}
Other choices are possible; the {  corresponding} constants in the coercivity bounds can be traced explicitly from the preceding estimates.
\end{remark}
\begin{remark}[Constants if  $\Omega$ is a square or a cube]
\label{remark_square}
{  To give a sense of scale,  in the simple case where $\Omega$ is a square in $\dimx = 2$  or a cube in $\dimx = 3$, of diameter $L$ centred at the 
origin, then $L_0$ is the radius of the inscribed circle that equals  $L_0 = \frac{L}{2\sqrt{\dimx}}$. Choosing the parameters 
$\varepsilon_2$, $\varepsilon_3$ as above and the rest of the parameters appearing in \eqref{coerc_2}, namely
$\alpha$, $\varepsilon_1$, $\beta$, $\gamma_1$ and $\gamma_2$, as follows,
yields the coercivity estimates:

\paragraph{$\dimx = 2$:} Parameters:
\[
   \alpha = 1, \quad \varepsilon_1 = \tfrac{1}{2}, \quad
   \varepsilon_2 = \tfrac{L_0}{4}, \quad
   \varepsilon_3 = \tfrac{L_0}{4L^2}, \quad
   \beta = 6.2 L, \quad
   \gamma_1 = 39.5 L^2, \quad
   \gamma_2 = 5.7 L.
\]
Then
\begin{align*}
  2 F_{\gamma} (u) \ge\ & 0.12 L^2 \Norm{\L  u}^2 + \frac{1}{2}\Norm{\nabla u}^2 + \frac{1}{2}k^2 \Norm{u}^2 \\
   & + \frac{L}{4\sqrt{2}} \Norm{\nabla u}_{\partial \Omega}^2 + [12.4 - (1 + 8\sqrt{2})] L k^2 \Norm{u}_{\partial \Omega}^2
      +  (11.4 - 8\sqrt{2})L \Norm{\partial_{\mathbf n}u - iku}^{2}_{\partial\Omega} .
\end{align*}

\paragraph{$\dimx = 3$:} Parameters:
\[
   \alpha = 1, \quad \varepsilon_1 = \tfrac{1}{4}, \quad
   \varepsilon_2 = \tfrac{L_0}{4}, \quad
   \varepsilon_3 = \tfrac{L_0}{4L^2}, \quad
   \beta = 7.5 L, \quad
   \gamma_1 = 26.5 L^2, \quad
   \gamma_2 = 5 L.
\]
Then
\begin{align*}
  3 F_{\gamma} (u) \ge\ & \frac{1}{2} L^2 \Norm{\L  u}^2 + \frac{1}{4}\Norm{\nabla u}^2 + \frac{3}{4}k^2 \Norm{u}^2 \\
   & + \frac{L}{4\sqrt{3}} \Norm{\nabla u}_{\partial \Omega}^2 + [15 - (1 + 8\sqrt{3})] L k^2 \Norm{u}_{\partial \Omega}^2
      +  (15 - 8\sqrt{3})L \Norm{\partial_{\mathbf n}u - iku}^{2}_{\partial\Omega} .
\end{align*}
In both cases, all coefficients of the boundary and interior terms are strictly positive, confirming the coercivity of  $F_\gamma(u)$
for these parameter choices.
}
\end{remark}

\section{Variational discretisation methods}
\label{sec:numerics}

This section describes the discrete formulations used in our numerical examples. We first consider a conforming finite element discretisation of the augmented variational principle introduced in Section~\ref{sec:variational}. The neural-network formulation is discussed in the following subsection. The purpose of the experiments is verification and illustration, a systematic computational comparison is left for future work.

\subsection{Finite element approximation}

\subsubsection*{Discrete formulation}

Let $V_h\subset \Vc$ be a finite-dimensional conforming space. Since the augmented formulation contains the bulk and the impedance residual, all terms are well defined if, for example, $V_h\subset H^2(\Omega)$. In the computations below we use $H^2$-conforming Argyris elements, although other $C^1$ finite element spaces could also be used.

The finite element approximation is obtained by restricting the stationary formulation associated with $\mathcal{F}_{\gamma,\C}$ to $V_h$: find $u_h\in V_h$ such that
\begin{equation}
\label{eq:fem_augmented_form}
\mathcal{A}_{\gamma,\C}(u_h,v_h)
=
\int_{\Omega} f \overline{\qp{v_h+2\gamma_1\L v_h}} \d x,
\qquad
\text{for all } v_h\in V_h,
\end{equation}
where $\mathcal{A}_{\gamma,\C}$ is the sesquilinear form defined in \eqref{Aform_reg_wbc}. 

\begin{proposition}[Well-posedness of the finite element problem]
\label{prop:fem_wellposed}
Assume that the hypotheses of Theorem~\ref{theorem_coerc_wbc} hold and let $V_h\subset\Vc$. If $\gamma_1\ge \gamma_{1,0}$ and $\gamma_2\ge \gamma_{2,0}$, then the finite element problem \eqref{eq:fem_augmented_form} admits a unique solution $u_h\in V_h$.
\end{proposition}

Indeed, coercivity of the real part of $\mathcal{A}_{\gamma,\C}$ on $\Vc$ implies coercivity on the subspace $V_h$, while boundedness follows directly from the definition of $\Vc$. The result is therefore an immediate finite-dimensional consequence of the complex Lax--Milgram theorem.

The same argument gives the usual quasi-optimality estimate in the augmented norm.

\begin{theorem}[Quasi-optimality in the augmented norm]
\label{thm:fem_quasioptimal}
Assume that the hypotheses of Theorem~\ref{theorem_coerc_wbc} hold, and let
$\gamma_1\ge \gamma_{1,0}$ and $\gamma_2\ge \gamma_{2,0}$. Let
$V_h\subset\Vc$ be a conforming finite element space. If $u\in\Vc$ denotes
the solution of \eqref{Helm_PDE} and $u_h\in V_h$ solves the discrete problem
\eqref{eq:fem_augmented_form}, then
\begin{equation}
\label{eq:fem_quasioptimal}
\|u-u_h\|_{\Vc}
\le
\frac{M_\gamma}{2\tilde c}
\inf_{w_h\in V_h}\|u-w_h\|_{\Vc},
\end{equation}
where $\tilde c$ is the coercivity constant in
\eqref{Energy_H_gamma_Re_4_wbc}, and $M_\gamma$ is the continuity constant of
$\mathcal{A}_{\gamma,\C}$ on $\Vc\times\Vc$.
\end{theorem}

This estimate is a stability and best-approximation result in the %augmented 
norm of $\Vc$. It should not be interpreted as a pollution-free estimate in  $L^2(\Omega)$ or $H^1(\Omega)$ norms. For this reason, the verification tests below also report relative $L^2$ errors, relative wavenumber-weighted $H^1$ errors and residual diagnostics.

In the finite element implementation, the complex-valued problem is written as a mixed real-valued problem for $(u_R,u_I)$. All integrals are assembled in Firedrake using quadrature degree $12$. Unless otherwise stated, the resulting sparse systems are solved by a direct LU factorisation using MUMPS through PETSc.

For the numerical tests we use a mesh-scaled boundary-residual weight
\begin{equation}
\gamma_{2,h}=\frac{\tau}{h},
\end{equation}
where $h$ is the mesh parameter and $\tau>0$ is fixed. Thus the boundary contribution is evaluated with $\gamma_2=\gamma_{2,h}$. This is a discrete stabilisation choice used to balance the boundary residual with the interior residual on the finite element space. Unless otherwise stated, we take $\gamma_1=2$ and $\tau=50$.

\subsubsection*{Manufactured-solution verification}

We first consider a manufactured-solution test on the unit square $\Omega=(0,1)^2$. The exact solution is
\begin{equation}
u_{\mathrm{ex}}(x,y)
=
\exp\qp{
ik\qp{(x-\tfrac12)^2+(y-\tfrac12)^2}
}.
\end{equation}
On each side of the square, the outward normal derivative satisfies
\begin{equation}
\partial_{\mathbf n}u_{\mathrm{ex}}-iku_{\mathrm{ex}}=0,
\qquad
\text{on } \partial\Omega.
\end{equation}
The corresponding right-hand side is defined by
\[
f_{\mathrm{ex}}=-\Delta u_{\mathrm{ex}}-k^2u_{\mathrm{ex}},
\]
namely
\begin{equation}
f_{\mathrm{ex}}(x,y)
=
\qp{
-4ik
+4k^2\qp{(x-\tfrac12)^2+(y-\tfrac12)^2}
-k^2
}
u_{\mathrm{ex}}(x,y).
\end{equation}

We report the relative $L^2(\Omega)$ error, the relative
$H^1_k(\Omega)$ error, the relative bulk residual
\begin{equation}
\frac{
\|-\Delta u_h-k^2u_h-f_{\mathrm{ex}}\|_{L^2(\Omega)}
}{
\|f_{\mathrm{ex}}\|_{L^2(\Omega)}
},
\end{equation}
and the relative impedance residual
\begin{equation}
\frac{
\|\partial_{\mathbf n}u_h-iku_h\|_{L^2(\partial\Omega)}
}{
\|\partial_{\mathbf n}u_h\|_{L^2(\partial\Omega)}
+k\|u_h\|_{L^2(\partial\Omega)}
}.
\end{equation}

\begin{table}[htbp]
\centering
\begin{tabular}{rrrrrr}
\hline
$h$ & DoFs & rel. $L^2$ & rel. $H^1_k$ & rel. bulk res. & rel. bnd. res. \tabularnewline
\hline
0.0625 & 5068 & 1.798e-03 & 2.150e-03 & 1.064e-02 & 5.431e-04 \tabularnewline
0.03125 & 19340 & 7.198e-06 & 2.533e-05 & 7.107e-04 & 2.806e-05 \tabularnewline
0.015625 & 75532 & 4.368e-08 & 6.698e-07 & 4.371e-05 & 1.027e-06 \tabularnewline
0.0078125 & 298508 & 2.154e-09 & 1.991e-08 & 2.719e-06 & 3.405e-08 \tabularnewline
\hline
\end{tabular}
\caption{Manufactured-solution mesh refinement for the augmented Helmholtz formulation at $k=25$.}
\label{tab:fe-mms-refinement}
\end{table}

\begin{table}[htbp]
\centering
\begin{tabular}{rrrrrrr}
\hline
$k$ & $h$ & DoFs & rel. $L^2$ & rel. $H^1_k$ & rel. bulk res. & rel. bnd. res. \\
\hline
10 & 0.015625 & 75532 & 8.096e-10 & 1.150e-08 & 1.638e-06 & 1.687e-08 \\
25 & 0.015625 & 75532 & 4.368e-08 & 6.698e-07 & 4.371e-05 & 1.027e-06 \\
50 & 0.015625 & 75532 & 2.520e-05 & 3.501e-05 & 6.632e-04 & 2.790e-05 \\
100 & 0.015625 & 75532 & 2.024e-02 & 2.220e-02 & 1.106e-02 & 6.038e-04 \\
\hline
\end{tabular}
\caption{Manufactured-solution wavenumber sweep for the augmented Helmholtz formulation.}
\label{tab:fe-mms-k-sweep}
\end{table}

Table~\ref{tab:fe-mms-refinement} shows the refinement behaviour at fixed wavenumber $k=25$. The relative $L^2$ and $H^1_k$ errors decrease under mesh refinement, and the bulk and boundary residuals decrease consistently. Table~\ref{tab:fe-mms-k-sweep} reports a fixed-mesh wavenumber sweep with $h=2^{-6}$. As expected, the errors increase as the wavenumber grows on a fixed mesh.

\subsubsection*{Localised-source illustration}

We next consider a qualitative test with a smooth source concentrated near the centre of the domain:
\begin{equation}
f(x,y)
=
\exp\qp{
-\frac{(x-0.5)^2+(y-0.5)^2}{\varepsilon}
},
\qquad
\varepsilon=10^{-4}.
\end{equation}
The same Argyris discretisation and parameters $\gamma_1=2$ and $\tau=50$ are used. The computations are performed on $\Omega=(0,1)^2$ using a uniform triangular mesh with $h\approx 2^{-7}$. Figure~\ref{fig:argyris-localised-source} shows the real part of the computed solution for increasing wavenumber.

\begin{figure}[h!]
\centering
\begin{subfigure}{0.24\linewidth}
\includegraphics[width=\textwidth]{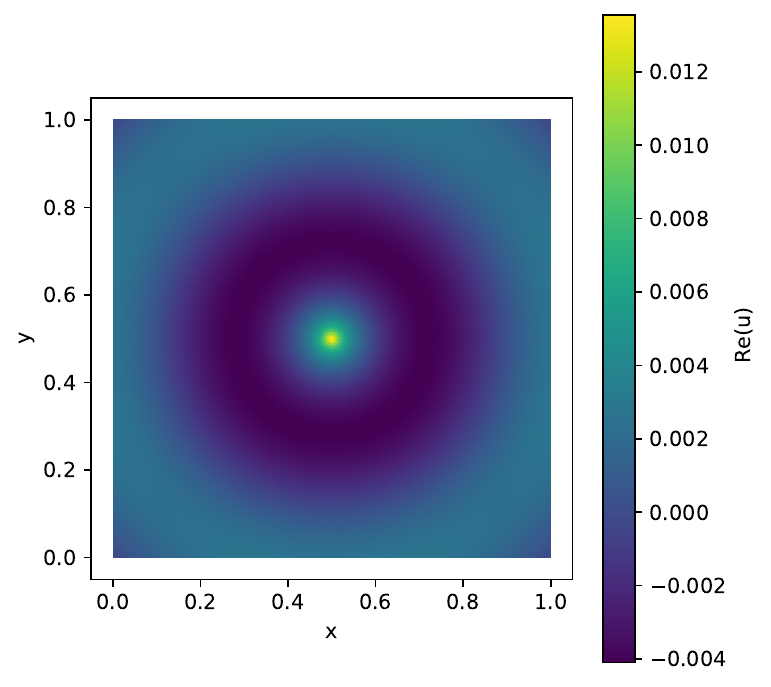}
\caption{$k=10$}
\end{subfigure}
\begin{subfigure}{0.24\linewidth}
\includegraphics[width=\textwidth]{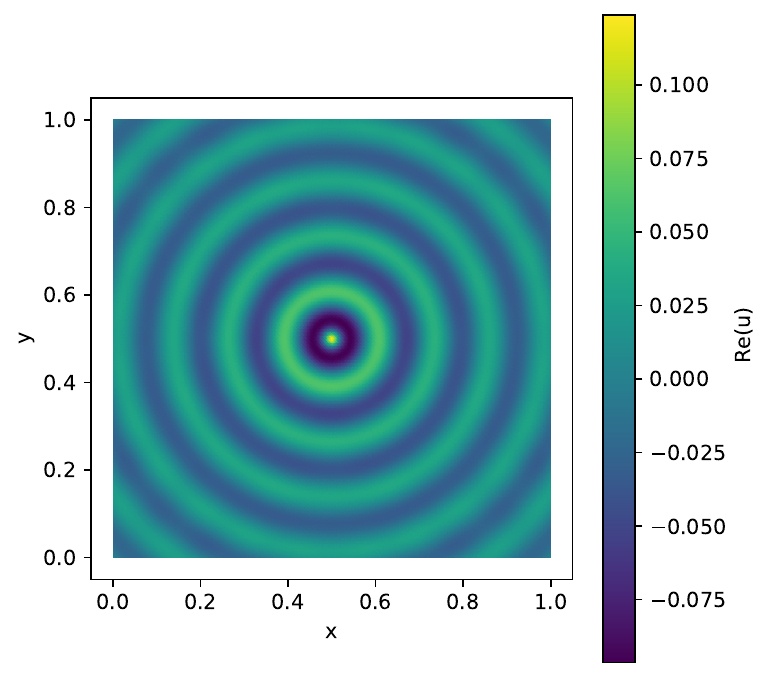}
\caption{$k=50$}
\end{subfigure}
\begin{subfigure}[b]{0.24\linewidth}
\includegraphics[width=\textwidth]{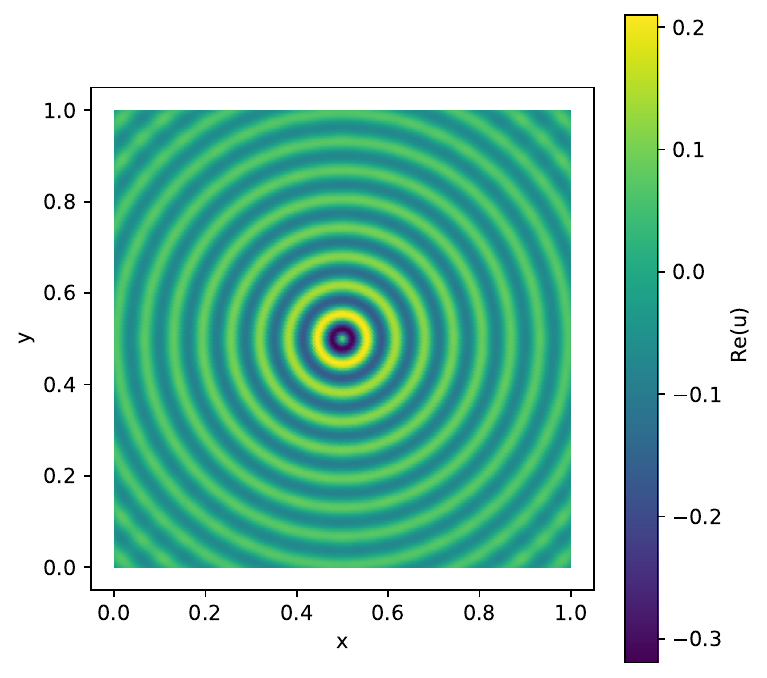}
\caption{$k=100$}
\end{subfigure}
\begin{subfigure}[b]{0.24\linewidth}
\includegraphics[width=\textwidth]{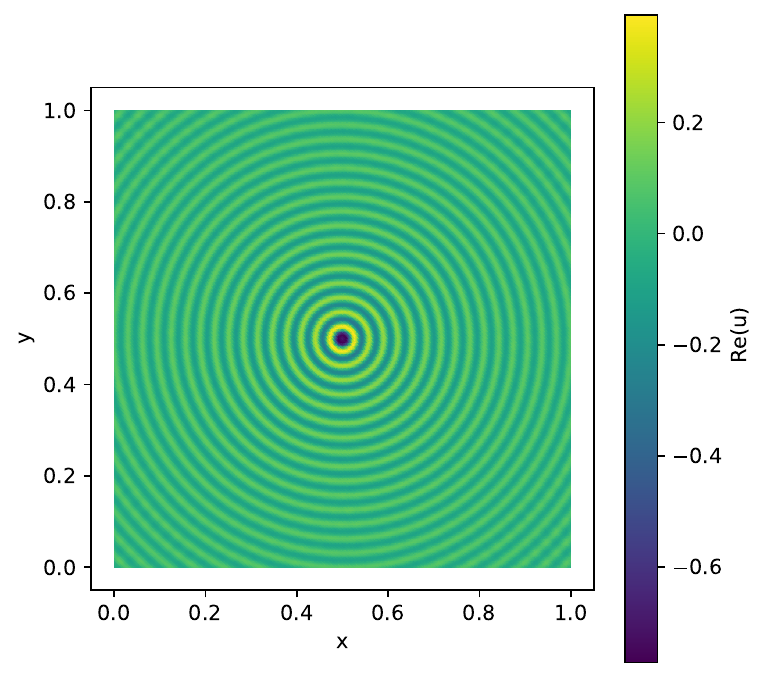}
\caption{$k=200$}
\end{subfigure}
\caption{Real part of the Argyris finite element solution on $\Omega=(0,1)^2$ with $h\approx 2^{-7}$ for $k\in\{10,50,100,200\}$.}
\label{fig:argyris-localised-source}
\end{figure}

\vspace{-5em}

\subsubsection*{Nonconvex scattering benchmark}
\label{subsec:nonconvex-scattering}

Finally, we include a nonconvex scattering benchmark. This example lies outside the star-shaped setting used in the coercivity proof, and is included as an illustrative stress test of the discrete formulation.

The computational domain is the unit square with a U-shaped sound-soft obstacle removed:
\begin{equation}
\Omega=(0,1)^2\setminus\overline{D_U},
\end{equation}
where
\begin{equation}
D_U
=
\qp{[0.38,0.45]\times[0.30,0.75]}
\cup
\qp{[0.55,0.62]\times[0.30,0.75]}
\cup
\qp{[0.38,0.62]\times[0.30,0.38]}.
\end{equation}
The geometry contains re-entrant corners and a cavity opening.

We prescribe the incident plane wave
\begin{equation}
u_{\mathrm{inc}}(x,y)
=
\exp\qp{ik d\cdot(x,y)},
\qquad
d=(0,-1),
\end{equation}
so that the wave is directed into the opening of the obstacle. The total field $u$ satisfies
\begin{equation}
-\Delta u-k^2u=0
\qquad
\text{in } \Omega.
\end{equation}
On the obstacle boundary $\Gamma_{\mathrm{obs}}:=\partial D_U$ we impose the sound-soft condition
\begin{equation}
u=0
\qquad
\text{on } \Gamma_{\mathrm{obs}}.
\end{equation}
On the exterior boundary $\Gamma_{\mathrm{out}}:=\partial(0,1)^2$ we impose the first-order absorbing condition on the scattered field. Equivalently, the total field satisfies
\begin{equation}
\partial_{\mathbf n}u-iku
=
\partial_{\mathbf n}u_{\mathrm{inc}}-iku_{\mathrm{inc}}
\qquad
\text{on } \Gamma_{\mathrm{out}}.
\end{equation}
In the implementation, the outer boundary residual is shifted by the known datum
\[
g_{\mathrm{out}}
:=
\partial_{\mathbf n}u_{\mathrm{inc}}-iku_{\mathrm{inc}},
\]
while the sound-soft condition on $\Gamma_{\mathrm{obs}}$ is imposed weakly through a mesh-scaled penalty term.

Figure~\ref{fig:u-obstacle-scattering} shows the augmented finite element solution for $k=40$ on a mesh with $h\approx0.045$. The incident wave enters the cavity from above, producing interference around the re-entrant corners and enhanced field intensity inside the U-shaped region.

\begin{figure}[htbp]
\centering
\includegraphics[width=\textwidth]{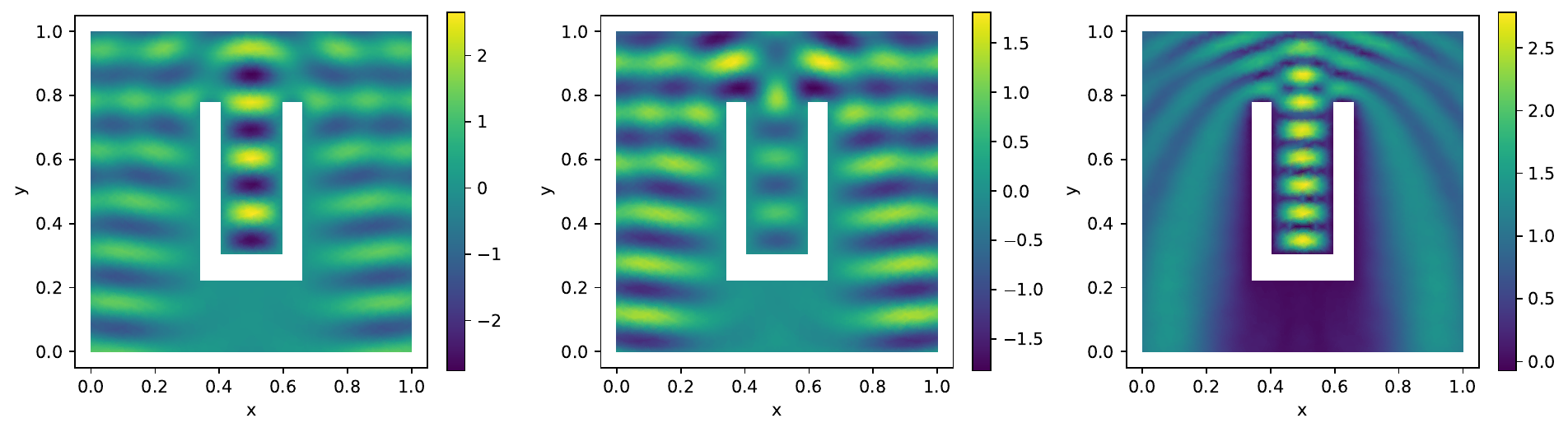}
\caption{Augmented finite element solution for the U-shaped sound-soft obstacle scattering problem with $k=40$ and $h\approx0.045$. The panels show the real, the imaginary part and amplitude of the computed total field.}
\label{fig:u-obstacle-scattering}
\end{figure}
\vspace{-4em}

\subsection{Neural network approximation}

\subsubsection*{Deep fixed point flux formulation}

We also consider a neural network discretisation of the augmented variational formulation. The main difference from the finite element method is that the impedance contribution is represented through a boundary flux variable rather than assembled directly into a conforming trial space.

We introduce a complex boundary variable $\lambda$ and write the weak Helmholtz equation as
\begin{equation}
\label{eq:nn_flux_weak_form}
\int_\Omega \nabla u\cdot\nabla\overline v \d x
-k^2\int_\Omega u\overline v \d x
-\int_{\partial\Omega}\lambda\overline v \d S
=
\int_\Omega f\overline v \d x .
\end{equation}
The impedance condition is encoded by the boundary relation
$\lambda=iku \qquad \text{on } \partial\Omega.$ At a fixed point, we
recover the impedance term appearing in the variational formulation.
For a fixed boundary flux $\lambda$, the primal network is obtained by
approximately minimising
\begin{equation}
\label{eq:nn_deep_uzawa_functional}
\begin{split}
\mathcal{J}_{\gamma,\lambda}(v)
&=
\frac12\int_\Omega\qp{|\nabla v|^2-k^2|v|^2}\d x
-\Re\int_\Omega f\overline v\d x
+\gamma_1\int_\Omega|\L v-f|^2\d x
\\
&\quad
+\gamma_2\int_{\partial\Omega}
|\partial_{\mathbf n}v-ikv|^2\d S
-\Re\int_{\partial\Omega}\lambda\overline v\d S .
\end{split}
\end{equation}
The last term supplies the weak flux contribution. Since this term is linear in $v$, it does not affect the coercivity of the quadratic part, see Theorem \ref{Theorem_SCVP_2}.

The complex-valued approximation is represented as
\[
u_\theta=u_{\theta,R}+iu_{\theta,I},
\]
where $u_{\theta,R}$ and $u_{\theta,I}$ are the two real-valued outputs of a SIREN network. Derivatives in $\L u_\theta$ and on $\partial\Omega$ are computed by automatic differentiation. On a fixed boundary collocation set $\{y_m\}_{m=1}^{M_\partial}\subset\partial\Omega$, the flux is updated by the damped {fixed point} step, compare to Theorem \ref{Theorem_SCVP_2} and 
\eqref{fixed_point_it_tilde},
\begin{equation}
\label{eq:nn_uzawa_update}
\lambda^{m+1}(y)
=
(1-\rho)\lambda^m(y)
+\rho\,ik\,u_{\theta^{m+1}}(y),
\qquad
y\in\partial\Omega,
\end{equation}
with $0<\rho\le 1$. Equivalently,
\[
\lambda_R^{m+1}
=
(1-\rho)\lambda_R^m-\rho k u_I^{m+1},
\qquad
\lambda_I^{m+1}
=
(1-\rho)\lambda_I^m+\rho k u_R^{m+1}.
\]

\subsubsection*{Implementation details}

All integrals in \eqref{eq:nn_deep_uzawa_functional} are approximated by Monte Carlo quadrature using independently sampled interior and boundary points, with the appropriate domain and boundary measure factors. Unless otherwise stated, we use a SIREN network with two outputs, width $96$, depth $4$, Adam optimisation and random seed $67$.

For the two-dimensional manufactured-solution tests we use $4096$ interior points and $2048$ boundary points per primal update. The Deep Fixed Point runs use $200$ outer iterations, $20$ primal Adam steps per outer iteration and a preliminary residual least-squares warm start of $500$ Adam steps all in single precision. The relaxation parameter is fixed at $\rho=0.2$.

As a baseline, we train a residual PINN with the same architecture, optimiser, sampling budget and approximately the same number of Adam updates. This baseline minimises the bulk residual and the impedance boundary residual directly, without the physical Helmholtz energy and without the flux update \eqref{eq:nn_uzawa_update}.

\subsubsection*{Manufactured-solution comparison}

We use the same two-dimensional manufactured solution as in the finite element verification test,
\[
u_{\mathrm{ex}}(x,y)
=
\exp\qp{
ik\qp{(x-\tfrac12)^2+(y-\tfrac12)^2}
},
\]
with $f_{\mathrm{ex}}=-\Delta u_{\mathrm{ex}}-k^2u_{\mathrm{ex}}$. The errors are evaluated on validation points not used in the corresponding training step. We report the relative $L^2(\Omega)$ error, the relative $H^1_k(\Omega)$ error, the relative bulk residual and the relative boundary residual when available.

\begin{table}[htbp]
\centering
\begin{tabular}{rlrrr}
\hline
$k$ & method & rel. $L^2$ & rel. $H^1_k$ & rel. bulk res. \\
\hline
10 & augmented & 2.552e-02 & 2.751e-02 & 2.244e-02 \\
10 & residual PINN & 6.503e-02 & 7.363e-02 & 4.564e-02 \\
\hline
25 & augmented & 5.036e-03 & 5.556e-03 & 3.583e-03 \\
25 & residual PINN & 2.193e-01 & 2.388e-01 & 6.820e-02 \\
\hline
50 & augmented & 1.226e-03 & 1.308e-03 & 9.853e-04 \\
50 & residual PINN & 5.825e-01 & 6.551e-01 & 3.293e-01 \\
\hline
100 & augmented & 1.370e-03 & 1.465e-03 & 1.061e-03 \\
100 & residual PINN & 9.621e-01 & 9.761e-01 & 8.796e-01 \\
\hline
\end{tabular}
\caption{Manufactured-solution wavenumber sweep comparing the augmented Deep Fixed Point formulation with a residual PINN baseline. }
\label{tab:nn-mms-k-sweep-augmented-vs-pinns}
\end{table}

Table~\ref{tab:nn-mms-k-sweep-augmented-vs-pinns} shows that the augmented formulation gives substantially smaller validation errors across this sweep. The comparison is merely a verification test for the proposed formulation not a complete performance study.

\subsubsection*{Three-dimensional obstacle scattering}
\label{subsec:nn-3d-obstacle-scattering}

We finally include a three-dimensional scattering example to
illustrate the neural formulation on a non-star-shaped geometry. The
computational domain is the unit cube with one or more vertical
cylindrical sound-soft obstacles removed,
$\Omega=(0,1)^3\setminus\overline D .$ In the reported configuration,
$D$ is a staggered pair of cylinders, $D = D_1 \cup D_2 ,$ where
\begin{equation}
D_j
=
\{(x,y,z)\in(0,1)^3:
(x-c_{j,1})^2+(y-c_{j,2})^2<R_j^2\},
\qquad j=1,2 .
\end{equation}
We prescribe the incident plane wave
\begin{equation}
u_{\mathrm{inc}}(\vec x)
=
\exp\qp{ik d\cdot \vec x},
\qquad
d=\frac{(0.55,-1,0)}{\sqrt{0.55^2+1}} ,
\end{equation}
so that the incoming wave is oblique in the $xy$-plane. The total
field $u$ satisfies $-\Delta u-k^2u=0$ in $\Omega .$ On the
cylindrical obstacle boundary $\Gamma_{\mathrm{obs}}:=\partial D$ we
impose the sound-soft condition $u=0$ on $\Gamma_{\mathrm{obs}}.$ This
condition is imposed strongly in the neural ansatz.

The scattered field is defined by $u_s=u-u_{\mathrm{inc}}.$ On the
four vertical faces of the cube we impose the first-order absorbing
condition on the scattered field, $\partial_{\vec n}u_s-iku_s=0 .$ On
the top and bottom faces, $z=0$ and $z=1$, we impose the corresponding
axial consistency condition $\partial_z u = \partial_z
u_{\mathrm{inc}},$ which reduces to a homogeneous Neumann condition
for the incident direction as $d_3=0$.

The reported run uses the same neural residual least-squares structure
as in the two-dimensional circular
benchmark. Figure~\ref{fig:nn-3d-obstacle-scattering} shows two
cutaway visualisations of the computed field. The plots display the
three-dimensional extruded scattering pattern generated by the oblique
incident wave interacting with the cylindrical sound-soft obstacles.

\begin{figure}[htbp]
\centering
\begin{subfigure}[t]{0.4\textwidth}
\centering
\includegraphics[width=\textwidth]{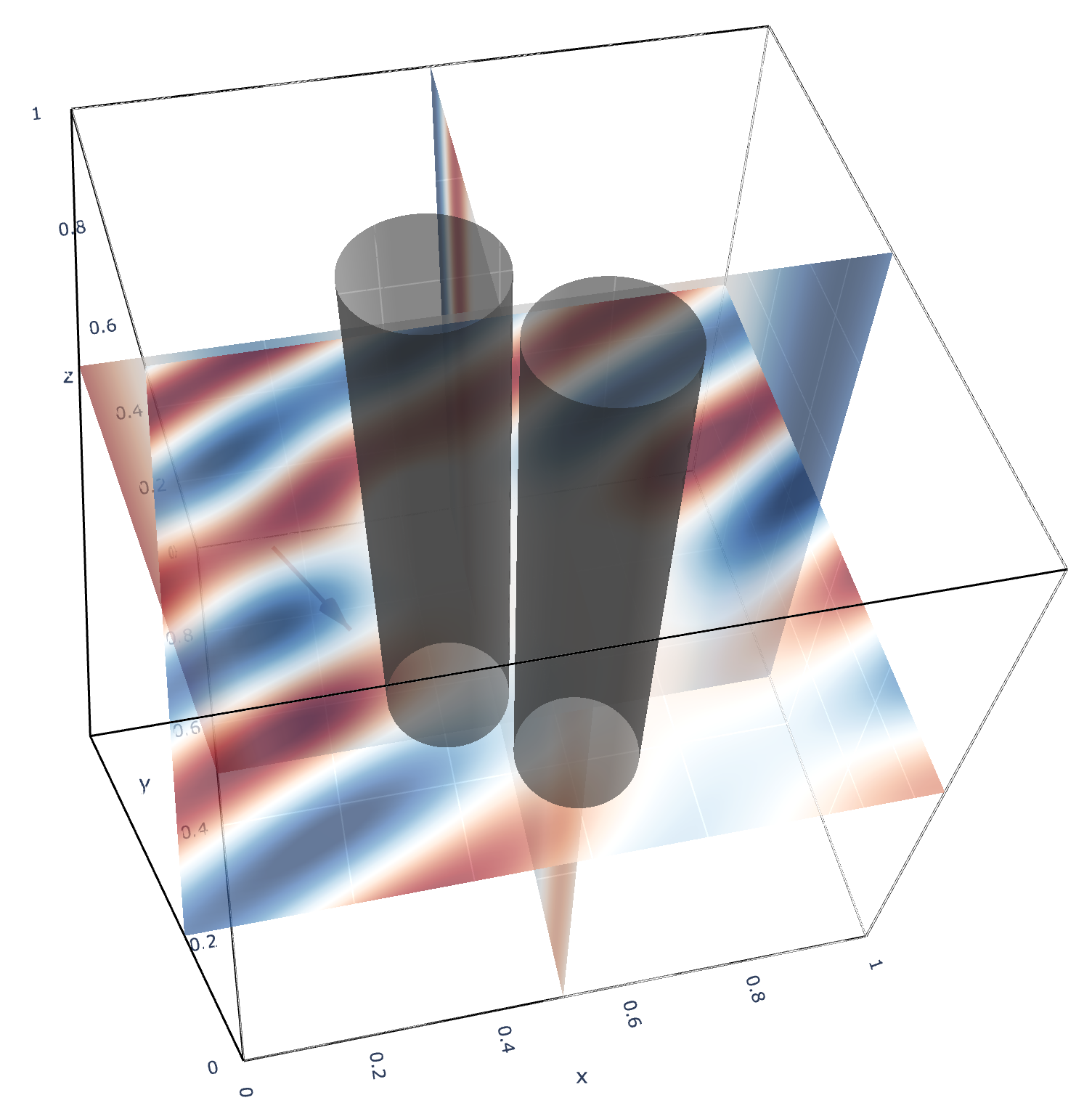}
\caption{Cutaway view of the computed three-dimensional field.}
\label{fig:nn-3d-obstacle-scattering-a}
\end{subfigure}
\hfill
\begin{subfigure}[t]{0.4\textwidth}
\centering
\includegraphics[width=\textwidth]{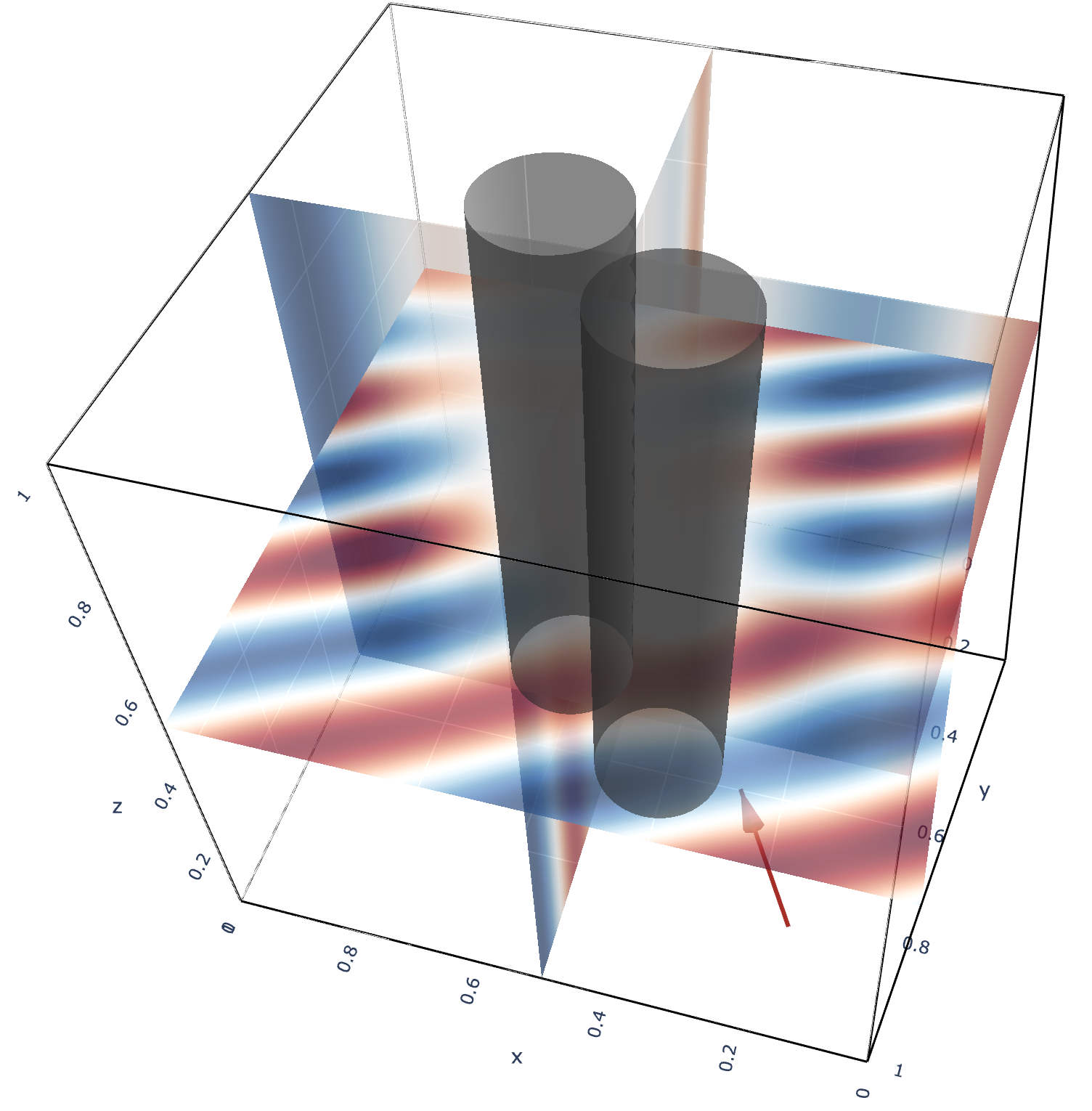}
\caption{Alternative cutaway view showing the interaction with the extruded obstacles.}
\label{fig:nn-3d-obstacle-scattering-b}
\end{subfigure}
\caption{Three-dimensional neural solution for the extruded cylindrical obstacle scattering benchmark with $k=30$. The computation is carried out in the unit cube with vertical sound-soft cylindrical obstacles removed. The sound-soft condition is imposed strongly through the neural ansatz, while the absorbing condition is imposed on the scattered field on the four vertical exterior faces. The arrow determines the incoming incident wave.}
\label{fig:nn-3d-obstacle-scattering}
\end{figure}

\printbibliography

\end{document}